\documentclass[letter,11pt]{article}
\usepackage{amsmath,amssymb,amsthm}
\usepackage{tikz,enumerate}
\usepackage{cite,mathtools}
\usepackage{etex}
\theoremstyle{plain}

\newtheorem{theorem}{Theorem} 
\newtheorem{proposition}{Proposition} 
\newtheorem{corollary}{Corollary}

\newtheorem{lemma}{Lemma}

\newtheorem{example}{Example}

\theoremstyle{remark}

\newtheorem{remark}{Remark}

\theoremstyle{definition}
\newtheorem{definition}{Definition}
\usepackage{mathrsfs,enumerate}
\usepackage[textheight=22 cm, textwidth=17cm, headheight=60pt, headsep=30pt, footskip=32pt, margin=1in]{geometry}

\usepackage[draft]{commenting}
\usepackage{array}
\newcolumntype{C}[1]{>{\centering\arraybackslash}m{#1}}

\title{On the Ratio of Shannon Numbers of Graphs}

\author{Sharareh Alipour, Amin Gohari,
 Mehrshad Taziki
}
\begin{document}

\maketitle

\begin{abstract}
Let $\Gamma$ be a function that maps two arbitrary graphs $G$ and $H$ to a non-negative real number  such that
$$\alpha(G^{\boxtimes n})\leq \alpha(H^{\boxtimes n})\Gamma(G,H)^n$$
where $n$ is any natural number and $G^{\boxtimes n}$ is the strong product of $G$ with itself 
$n$ times. We establish the equivalence of two different approaches for finding such a function $\Gamma$. The common solution obtained through either approach is termed ``the relative fractional independence number of a graph $G$ with respect to another graph $H$". We show this function by $\alpha^*(G|H)$ and discuss some of its properties. In particular, we show that
$\alpha^*(G|H)\geq \frac{X(G)}{X(H)} \geq \frac{1}{\alpha^*(H|G)},$
where $X(G)$ can be the independence number, the  Shannon capacity, the fractional independence number, the Lov\'{a}sz number, or the Schrijver's or Szegedy's variants of the Lov\'{a}sz number of a graph $G$. This inequality is the first explicit non-trivial upper bound on the ratio of the invariants of two arbitrary graphs, as mentioned earlier, which can also be used to obtain upper or lower bounds for these invariants. As explicit applications, we present new upper bounds for the ratio of the  Shannon capacity of two Cayley graphs and compute new lower bounds on the Shannon capacity of certain Johnson graphs (yielding the exact value of their Haemers number). Moreover, we show that $\alpha^*(G|H)$ can be used to present a stronger version of the well-known No-Homomorphism Lemma. %The No-Homomorphism Lemma is widely used to show the non-existence of a homomorphism between two graphs and is also used to give an upper bound on the independence number of a graph. %Our No-Homomorphism Lemma is computationally more accessible than its original version.
\footnote{A shortened conference version of this work (containing some of the results) was presented at the 2024 IEEE Information Theory Workshop.}
\end{abstract}


%%
%% Bibliography
%%

%% Please use bibtex, 

\section{Introduction}
Let $G$ be a finite, undirected graph without a loop or multiple edges. The set $\mathcal V(G)$ denotes the vertex set of $G$, and $\mathcal E(G)$ represents the edge set of $G$. An independent set of $G$ is a subset $\mathcal S$ of $\mathcal V(G)$ such that no two vertices in $\mathcal S$ are adjacent in $G$, i.e., $uv\notin \mathcal E(G)$ for all $u,v\in \mathcal S$. The \emph{independence number} (or the \emph{packing number}) of a graph $G$, denoted by $\alpha(G)$, is the size of the biggest independent set of $G$. 
 Computing $\alpha(G)$ is an NP-hard problem \cite{karp}. 
The strong product of two graphs, $G \boxtimes H$, is a graph whose vertex set is the Cartesian product of the vertex sets of $G$ and $H$. Distinct vertices $(u,u')$ and $(v,v')$ are adjacent in $G \boxtimes H$ iff either $u=v$ and $u'v'\in \mathcal E(H)$, or $uv\in \mathcal E(G)$ and $u'=v'$, or $uv\in \mathcal E(G)$ and $u'v'\in \mathcal E(H)$. 

The Shannon capacity (or the Shannon number) of a graph $G$ is defined as
\begin{align}
 \mathscr{C}(G)
 = \lim_{n \rightarrow \infty}  \alpha(G^{\boxtimes n})^{\frac{1}{n}},\label{eqnDefC}
\end{align}
where $G^{\boxtimes n}$ is the strong graph product of $G$ with itself $n$ times. It is known that the above limit exists and furthermore $\mathscr{C}(G)\geq \alpha(G^{\boxtimes n})^{\frac{1}{n}}$ for every $n\ge 1$. 
This graph invariant was introduced by Shannon in 1956 and
gives a measure of the optimal zero-error performance of an associated memory-less communication channel \cite{shannonmain}.

In general, computing the exact value of $\mathscr{C}(G)$ is a challenging problem, even for simple graphs like a cycle of length $7$, $C_7$, $\mathscr{C}(C_7)$ is not known.
%Lov\'{a}sz proved that $\mathscr{C}(C_5)=\sqrt{5}$, where $C_5$ is a cycle of length $5$. For more results on the Shannon capacities of odd cycles or their compliments, see \cite{bohman,mathew,bohman2003nontrivial}. 
%See also 
%\cite{sason2023observations}.
%Alon and Lubetzky in \cite{alon} showed that the series of independence numbers in strong powers of a fixed graph can exhibit a complex structure, implying that the Shannon capacity of a graph cannot be approximated (up to a sub-polynomial factor of the number of vertices) by any arbitrarily large, yet fixed, prefix of the series of independence numbers in strong powers. 
Nonetheless, various upper bounds on the Shannon capacity are obtained in the literature. This can be achieved by identifying a function $\Lambda$ of a graph $G$ such that
\begin{align}
    \alpha(G^{\boxtimes n})\leq \Lambda(G)^n\label{Lambdaeq}
\end{align}
for any natural number $n$. Examples of such upper bounds are the fractional independence number,  Lov\'{a}sz number of a graph \cite{Lovasz} or the Haemers number \cite{haemers}.
The Lov\'{a}sz number can be formulated as a semidefinite program and numerically approximated by the ellipsoid method in polynomial time in the number of vertices of $G$ \cite{gro}.
The Haemers number of a graph $G$, $\mathcal{H}(G)$, is an upper bound on $\mathscr{C}(G)$ which considers the rank of particular matrices associated with the graph $G$ \cite{haemers}.
An upper bound on the Shannon capacity of a graph via a linear
programming variation is given in \cite{hu}.
In \cite{bukh}, a fractional version of Haemers
bound is presented, and it is shown that this fractional version outperforms the Haemers bound. It is generally challenging to compute the exact values of Haemers number and the fractional Haemers number of a given graph $G$. However, these values are computed exactly \cite{haemers, riddle} for some special graphs, such as Johnson graphs with certain parameters.
Please see \cite{alon2002graph,jurkiewicz2014survey,korner1998zero,sason2023observations}  for more related results and survey papers.

In this paper, we are interested in a variant of the property in \eqref{Lambdaeq} for \emph{two} graphs. In particular, we seek a function  $\Gamma$ of two arbitrary graphs $G$ and $H$ such that
\begin{align}
    \alpha(G^{\boxtimes n})\leq \alpha(H^{\boxtimes n})\Gamma(G,H)^n \label{eqnv31d}
\end{align}
for any natural number $n$. 
Observe that \eqref{eqnv31d} implies that
\begin{align}
\Gamma(G,H)\geq \sup_n\left(\frac{\alpha(G^{\boxtimes n})}{\alpha(H^{\boxtimes n})}\right)^{\frac1n}\geq \lim_{n\rightarrow\infty}\left(\frac{\alpha(G^{\boxtimes n})}{\alpha(H^{\boxtimes n})}\right)^{\frac1n}=\frac{\mathscr{C}(G)}{\mathscr{C}(H)}\label{eqnv31dNN}
\end{align}
To the best of our knowledge, the problem of finding functions $\Gamma$ satisfying \eqref{eqnv31d} has not been studied before for general graphs. We point out that the No-Homomorphism Lemma partially addresses this question for a special class of graphs $G$ and $H$. This lemma states that if there is a homomorphism from $H$ to $G$, and $G$ is vertex-transitive, then
$$\frac{\alpha(G)}{\alpha(H)}\leq \frac{|\mathcal V(G)|}{|\mathcal V(H)|}.$$
It is clear that if there is a homomorphism from $H$ to $G$, there will also be a homomorphism from $H^{\boxtimes n}$ to $G^{\boxtimes n}$. Therefore,
$$\frac{\alpha(G^{\boxtimes n})}{\alpha(H^{\boxtimes n})}\leq \left(\frac{|\mathcal V(G)|}{|\mathcal V(H)|}\right)^n, \qquad\qquad\forall n\in\mathbb{N}.$$
Consequently, 
\begin{align}
\Gamma(G,H)=\frac{|\mathcal V(G)|}{|\mathcal V(H)|}\label{nohomeqn3}
\end{align}
will satisfy \eqref{eqnv31d} if there is a homomorphism from $H$ to $G$, and $G$ is vertex-transitive. However, we are interested in functions $\Gamma(G,H)$ which satisfy \eqref{eqnv31d} for \emph{all} graph pairs $G$ and $H$.

We discuss two different approaches for constructing such a function $\Gamma$.
The first approach yields a function
satisfying \eqref{eqnv31d} as the solution of an optimization problem involving a minimum, while the second approach yields a function satisfying \eqref{eqnv31d} in terms of a maximum.  Rather surprisingly, we show that the two approaches are equivalent and correspond to the dual forms of a linear programming problem. For reasons that will become clear later, the common solution obtained through either approach is named ``the relative fractional independence number of a graph $G$ with respect to another graph $H$", and is denoted by 
$$\alpha^*(G|H).$$
 For the special case of $G$ being a vertex-transitive graph, $\alpha^*(G|H)$ has a simple characterization: 
\begin{align*}\alpha^*(G|H)=\frac{|\mathcal V(G)|}{\alpha(G^c\boxtimes H)}.\end{align*}
Moreover, if in addition to $G$ being a vertex-transitive, there is also a homomorphism  from $H$ to $G$, then we show that 
\begin{align*}\alpha^*(G|H)=\frac{|\mathcal V(G)|}{\alpha(G^c\boxtimes H)}\leq \frac{|\mathcal V(G)|}{|\mathcal V(H)|},\end{align*}
so $\alpha^*(G|H)$ is less than or equal to \eqref{nohomeqn3} (the value proposed by the No-Homomorphism Lemma).

If both $G$ and $H$ are vertex-transitive graphs, we can also write
\begin{align}\alpha^*(G|H)=\frac{\chi_f(G^c\boxtimes H)}{ |\mathcal{V}(H)|},\end{align}
where $\chi_f(\cdot)$ is the fractional coloring number. We also obtain the explicit value of $\alpha^*(G|H)$ for some structured graphs (see Table \ref{table1}).

Our characterization of $\alpha^*(G|H)$ for arbitrary graphs $G$ and $H$ leads to the first explicit non-trivial upper bound on the ratio of the Shannon capacity of two arbitrary graphs. {In particular, for a vertex-transitive graph $G$ and any arbitrary graph $H$ we obtain
\begin{align}\frac{\mathscr{C}(G)}{\mathscr{C}(H)}\leq \frac{|\mathcal V(G)|}{\alpha(G^c\boxtimes H)}.\label{eqnCGR}\end{align}

Next, we discuss the relation of $\alpha^*(G|H)$ with other invariants, such as the Lov\'{a}sz number of a graph.
For a given graph $G$, using the fractional independence number of $G$ with respect to another graph $H$, we can give upper or lower bounds on the value of these invariants or compute some previously unknown invariants. 

One application of our results is the extension of the No-Homomorphism Lemma. Assume that there is a homomorphism from $H$ to $G$, and $G$ is vertex-transitive. Then, the No-Homomorphism Lemma yields 
$$\frac{\alpha(G^{\boxtimes n})}{\alpha(H^{\boxtimes n})}\leq \left(\frac{|\mathcal V(G)|}{|\mathcal V(H)|}\right)^n, \qquad\qquad\forall n\in\mathbb{N}.$$
The above equation also immediately implies that
$$\frac{\mathscr{C}(G)}{\mathscr{C}(H)}\leq \frac{|\mathcal V(G)|}{|\mathcal V(H)|}.$$
However, what if we replace $\mathscr{C}(G)$ and $\mathscr{C}(H)$ by the Lov\'{a}sz numbers of $G$ and $H$ (or say, by the fractional independence numbers of $G$ and $H$)? Utilizing the relative fractional independence number we can also show that
$$\frac{X(G)}{X(H)}\leq \alpha^*(G|H) \leq \frac{|\mathcal V(G)|}{|\mathcal V(H)|}$$
where $X(\cdot)$ can be either the independence number, the Shannon capacity, the fractional independence number, the Lov\'{a}sz number, Schrijver's
or Szegedy's variants of the Lov\'{a}sz number of the graph $G$. 

Finally, there are also other results included in the paper, sometimes serving as examples for our discussions. In particular, we find new lower bounds on the Shannon capacity of certain Johnson graphs, enabling us to determine the exact value of their Haemers number for the first time. We also establish that
$$\lim_{n\rightarrow \infty} 4n+2-\mathscr{C}(J(4n+2,3))=0$$
for the Johnson graph $J(4n+2,3)$.

\section{Notation and Preliminaries }

Throughout this paper, we use capital letters such as $G, H$, and $W$ to denote graphs (finite, undirected graphs with no loops or multiple edges). %As mentioned, $\mathcal V(G)$ and $\mathcal E(G)$ denote the vertex and edge set of $G$. 
 The complement of a graph $G$, denoted by $G^c$, is a graph with the same vertices as in $G$, such that two distinct vertices of $G^c$ are adjacent iff they are not adjacent in $G$. We use $C_k$ to denote a cycle graph of length $k$. We show sets in calligraphic letters. We use the lowercase letters to either denote vertices of a graph (as in $u,v,v_1, v_2, x, y$) or real numbers (as in $w_1, w_2, ...$). The bold letter $\mathbf{w}$ denotes a vector of real numbers. 
For graphs $G_1$ and $G_2$, $G=G_1+G_2$ is the disjoint union of $G_1$ and $G_2$, i.e., a graph whose vertex set is the union of vertex sets of $G_1$ and $G_2$ and whose edge set consists of edges which are in $G_1$ or in $G_2$.

We say that $G$ is an induced subgraph of $H$ if $G$ is formed by selecting a subset of the vertices from $H$ and then including \emph{all} the edges from the parent graph $H$ that connect those chosen vertices. Given a graph $G$, let $\mathcal I(G)$ be the set of all independent sets of $G$. We also consider the empty set as a member of $\mathcal I(G)$. Finally, a homomorphism $\mathsf{g}:H\rightarrow G$ from a graph $H$ to a graph $G$ is a map $\mathsf{g}:\mathcal V(H)\rightarrow \mathcal V(G)$ such that $uv\in \mathcal E(H)$ implies $\mathsf{g}(u)\mathsf{g}(v)\in \mathcal E(G)$. 

A graph is called a vertex-transitive (or node symmetric) graph if and only if for its any pair of nodes $v$ and $u$, there exists an
automorphism\footnote{An automorphism of a graph is a graph isomorphism with itself, i.e., a mapping from the vertices of the given graph $G$ back to vertices of $G$ such that the resulting graph is isomorphic with $G$.} of the graph that maps $v$ to $u$ \cite{chiang}. If $G$ and $H$ are vertex-transitive graphs, then so are $G^c$ and $G\boxtimes H$ (see \cite[Lemma 1]{SONNEMANN1974133}).

We also need the definition of a Cayley graph:
\begin{definition}
    Let $G$ be an abelian group and $S$ be a subset of elements of $G$ such that $S=-S$ and $0\not\in S$. The Cayley graph $Cay(G,S)$ is a graph with the vertex set $G$ such that $a,b\in G$ are connected by an edge iff $a-b \in S$. 
    Cayley graphs are basic examples of vertex-transitive graphs.
\end{definition}

% \begin{definition}[\cite{SONNEMANN1974133}]
% Given a graph $G$, we define $G^*$ as the \emph{independence graph} corresponding to $G$ as follows: vertices of $G^*$ are independent sets of $G$, i.e., $\mathcal{V}(G^*)=\mathcal I(G)$.
% Two independent sets $\mathcal S,\mathcal T\subset \mathcal V(G)$ are connected by an edge in $G^*$ if 

% disconnected in $G$ if $\mathcal S\cap \mathcal T=\emptyset$ and
% {
% $\mathcal S\cup \mathcal T\in \mathcal I(G)$ (i.e., there is no $u\in \mathcal S$ and $v\in \mathcal T$ such that $uv\in \mathcal E(G)$).
% }
% For example, if $G=C_7$, with vertex set $\{v_1,\dots, v_7\}$ and $\mathcal S=\{v_1,v_3\}$ and $\mathcal T=\{v_5\}$, then $\mathcal S$ and $\mathcal T$ are disconnected.

% \end{definition}

% {
% \begin{definition}
%     Given a graph $G$, we define $G^*$ as the \emph{independence graph} corresponding to $G$ as follows (please check that I write it correctly). The vertices of $G^*$ are independent sets of $G$. Two independent sets $S_1$ and $S_2$ (for $S_1\neq S_2$) are connected by an edge if the sets $S_1$ and $S_2$ are disconnected (according to our Definition 1) in $G$.
% \end{definition}
% }

\begin{definition}
We say that $\mathcal S,\mathcal T\in \mathcal I(G)$ are disconnected in $G$ if $\mathcal S\cap \mathcal T=\emptyset$ and 
$\mathcal S\cup \mathcal T\in \mathcal I(G)$.
% For example, if $G=C_7$, with vertex set $\{v_1,\dots, v_7\}$ and $\mathcal S=\{v_1,v_3\}$ and $\mathcal T=\{v_5\}$, then $\mathcal S$ and $\mathcal T$ are disconnected.
\label{def2sep}
\end{definition}
\begin{definition}\label{defG*}
    Following \cite{SONNEMANN1974133}, for a given graph $G$, the independence graph corresponding to $G$, denoted by $G^*$, is a simple graph whose vertices are elements of $\mathcal I(G)$ and two vertices are nonadjacent in $G^*$ iff their corresponding independent sets are disconnected. Equivalently, there is an edge between $\mathcal S,\mathcal T\in \mathcal I(G)$ where $\mathcal S\neq \mathcal T$
    if and only if $\mathcal S\cap \mathcal T\neq \emptyset$ or $\mathcal S\cup\mathcal T\notin\mathcal I(G)$.\footnote{There is also a different definition of independence graph in the literature which we do not use here \cite{brevsar2003independence}. } Note that $\emptyset$ is a vertex in $\mathcal I(G)$ and is not connected to any other vertex in $\mathcal I(G)$.
\end{definition}

\begin{table}
    \centering
    \renewcommand{\arraystretch}{1.5}
    \begin{tabular}{ |C{5cm}|C{3.8cm}|C{1.7cm}|C{4cm}| }
      \hline
      $\mathbf{G}$ & $\mathbf{H}$ & $\mathbf{\alpha^*(G|H)}$ & \textbf{Remarks} \\ [0.5ex]
      \hline\hline
      & & & Theorem \ref{cayex} \\
      $Cay(\mathbb{Z}_n, \pm 1,\dots \pm k)$ & $Cay(\mathbb{Z}_m, \pm 1,\dots, \pm k)$ & $n/m$ & $1\le 2k<n<m$ \\
      & & & $m=\ell n+s(k+1)$ \\ & & & for integers $\ell,s\ge 0$ \\
      \hline
      $Cay(\mathbb{Z}_n, \pm 1,\cdots,\pm (2k+1))$ & $Cay(\mathbb{Z}_n, \pm 1,\cdots,\pm k)$ & $1/2$ & Example \ref{example1d}, $n\geq 4k+3$ \\
      \hline
      $Cay(\mathbb{Z}_{n_1}, \pm 1,\dots, \pm (r+1))$ & $Cay(\mathbb{Z}_{n_2}, \pm 1, \dots, \pm r)$ & 1 & Example \ref{example2d}\\
      $n_1=(r+2)k+r$&$n_2=(r+1)k+r$&&$r,k\in\mathbb{N}$
      \\
      \hline
      $C_{2k+1}$ & $J(4k+2,3)$ & $1/4$ & Example \ref{examassig2}
       \\
      \hline
      $C_{2k}$ & $J(4k-1,3)$ & $k/(4k-3)$ & Example \ref{examassig2}\\
      &&&$k> 3$
      \\
      \hline
      & & & Appendix A, $n,m>1$ \\
      $C_n$ & $C_m$ & $n/(m-1)$ & if $n$ even and $m$ odd \\
      & & $n/m$ & if $m$ even \\
      & & $n/m$ & if $n,m$ odd and $n\leq m$ \\
      & & $n/(m-1)$ & if $n,m$ odd and $n> m$ \\
      \hline
    \end{tabular}
    \caption{Values of $\alpha^*(G|H)$ for some graph pairs; $Cay$ is the Cayley graph, $C_n$ is a cycle of length $n$ and $J(n,k)$ is the Johnson graph.}
    \label{table1}
\end{table}

\begin{definition}
    Given two arbitrary graphs $H$ and $G$, 
    let $\mathcal{F}(H,G)$ be the set of functions $f:\mathcal{V}(G)\rightarrow\mathcal{I}(H)$ with the following property: 
$f(v_1)$ and $f(v_2)$ are disconnected in $H$ if there is no edge between $v_1$ and $v_2$ in $G$.
    Equivalently, let $\hat{H}$ be $(H^*)^c$ with an added loop for the vertex corresponding to $\emptyset$ to itself. Then, $\mathcal{F}(H,G)$ is the class of all
homomorphisms from $G^c$ to $\hat H$.  
\end{definition}
\begin{remark}
    If a homomorphism $\mathsf{g}$ from $H$ to $G$ exists, for any vertex $w\in \mathcal V(G)$, the set $\{v\in \mathcal V(H): \mathsf{g}(v)=w\}$  must be an independent set in $H$ (empty set is considered to be an independent set). Thus, the ``inverse" of the homomorphism is a mapping $\mathsf{g}^{-1}: \mathcal V(G)\rightarrow \mathcal{I}(H)$ with the following properties: \begin{enumerate}[(i)]
    \item $\mathsf{g}^{-1}(v_1)$ and $\mathsf{g}^{-1}(v_2)$ are disconnected in $H$ if there is no edge between $v_1$ and $v_2$ in $G$, 
    \item The sets $\{\mathsf{g}^{-1}(v)\}$ for $v\in \mathcal{V}(G)$ form a partition of $\mathcal{V}(H)$.
\end{enumerate}
If we relax the condition (ii) above, and just keep (i), we recover the definition of $\mathcal{F}(H,G)$.

\end{remark}

\section{Two approaches for finding functions $\Gamma(G,H)$ satisfying
\eqref{eqnv31d}}

\subsection{The first approach}

Let us begin with the naive version of the first approach. 
Let $G$ be a graph with $k$ vertices $\mathcal{V}(G)=\{1,2,\cdots, k\}$. Assume that we have a collection of $k$ independent sets in $H$, denoted by $\mathcal T_1, \cdots, \mathcal T_k\in \mathcal I(H)$. We associate $\mathcal{T}_i$ with the $i$-th vertex in $G$. Independent sets $\mathcal T_1, \cdots, \mathcal T_k\in \mathcal I(H)$ satisfy the following property: $ \mathcal T_i$ and $\mathcal T_j$ are disconnected in $H$ if there is no edge between $v_i$ and $v_j$ in $G$ (see Definition \ref{def2sep}). Then, we claim that 
$$\alpha(H)\geq \alpha(G)\min_{v\in\mathcal V(G)}|\mathcal{T}_v|.$$
The reason is that for any independent set $\mathcal{S}$ in $G$, 
$\bigcup_{i\in\mathcal{S}}\mathcal{T}_i$
is an independent set in $H$ of size
$$\sum_{i\in\mathcal{S}}|\mathcal{T}_i|\geq |\mathcal{S}|\min_{v\in\mathcal V(G)}|\mathcal{T}_v|.$$

More generally, for any natural number $n$, we have
$$\alpha(H^{\boxtimes n})\geq \alpha(G^{\boxtimes n})\left(\min_{v\in\mathcal V(G)}|\mathcal{T}_v|\right)^n$$
because to a vertex $(v_1,v_2,\cdots,v_n)\in G^{\boxtimes n}$, we can assign the independent set $\mathcal{T}_{v_1}\times \mathcal{T}_{v_2}\times ...\times \mathcal{T}_{v_n}$ in $H^{\boxtimes n}$. 

The smallest $\Gamma(G,H)$ which can be obtained using the above approach is as follows:
$$\Gamma_0(G,H)=\min_{f} \max_{v\in\mathcal V(G)} \frac{1}{|f(v)|}$$
where the minimum is over deterministic functions $f\in \mathcal{F}(H,G)$. To see this, observe that functions $f\in \mathcal{F}(H,G)$ are equivalent with the assignment of independent sets $\mathcal{T}_v$ to the vertices of $G$. We set $\Gamma_0(G,H)=\infty$ if $\min_v|f(v)|=0$ for all functions $f\in \mathcal{F}(H,G)$.

By the above discussion, the function $\Gamma_0(G,H)$ satisfies 
\begin{align}
    \alpha(G^{\boxtimes n})\leq \alpha(H^{\boxtimes n})\Gamma_0(G,H)^n.
\end{align}
The utility of this approach is illustrated in the following example:
\begin{example}\label{examassig1}
Let $J(n,3)$ be the Johnson graph whose vertices are the $3$-subsets of $\{1,\dots, n\}$, and two vertices are adjacent if their intersection has one element.\footnote{We use the definition of \cite{haemers}, Some papers call the complement of this graph the Johnson graph.}
Haemers \cite{haemers} showed that $\mathcal H(J(n,3))\leq n$ and the equality holds for $n=4k$. 
Using the above approach, we compute a new lower bound for $\mathscr{C}(J(n,3))$ for $n=14,18,22,26$. This new lower bound implies that $\mathcal H(J(n,3))= n$ for $n=14,18,22,26$, which is unknown in the literature.\footnote{For $n=4k$, if we partition the underlying $n$-set into classes of size four, then all $3$-subsets, which are subsets of one of these classes, form an independent set of size $n$ in $J(n,3)$. Hence, $\alpha(J(n,3))\geq n$ when $n$ is divisible by $4$. Since $\mathcal H(J(n,3))\leq n$, \cite{haemers} deduces that $\alpha(J(n,3))=\mathscr C(J(n,3))=\mathcal H(J(n,3))$.
However, when $n$ is not divisible by $4$, the exact value of $\mathcal H(J(n,3))$ was unknown in \cite{haemers}.}

Let $C_n$ be a cycle of size $n$. Since $\alpha(C_{2k+1})=k$ and $\alpha (J(4k+2,3))=4k$ \cite{nagy1972certain}, we obtain:
$$\Gamma_0(C_{2k+1},J(4k+2,3))\geq \frac{\alpha(C_{2k+1})}{\alpha(J(4k+2,3))}=\frac{1}{4}. $$ 
    
    Consider the following assignment of the independent sets of $J(4k+2,3)$ to the vertices of $C_{2k+1}$. We assign the set $\mathcal T_i$ to $v_i$ for $1\leq i\leq 2k+1$ as follows: for $i\neq 2k+1$, let $\mathcal T_i=$ all the $3$-subsets of $\{2i-1, 2i, 2i+1,2i+2\}$ and $\mathcal T_{2k+1}=$ all the $3$-subsets of $\{4k+1, 4k+2, 1, 2\}$. 
This assignment shows that
$$\Gamma_0(C_{2k+1},J(4k+2,3))\leq\frac{1}{4}.$$
Therefore,
$$\Gamma_0(C_{2k+1},J(4k+2,3))=\frac{1}{4}.$$
Consequently, from \eqref{eqnv31dNN} we obtain
$$\frac{\mathscr C(C_{2k+1})}{\mathscr C(J(4k+2,3))}\leq \frac{1}{4}
.$$
Therefore, $\mathscr C(C_{2k+1})\leq \frac{1}{4}\mathscr C(J(4k+2,3))\leq \frac{1}{4} \mathcal H(J(4k+2,3))$, yielding the lower bounds $\mathscr C(J(4k+2,3))\geq 4\mathscr C(C_{2k+1})$. We also deduce that $\mathcal H(J(4k+2,3))\geq \lceil 4\mathscr C(C_{2k+1})\rceil$ where we can take the ceiling as $\mathcal H(J(n,3))$ is an integer number. Since $\mathscr C(C_7)\geq 3.2578$ \cite{polak}, $\mathscr C(C_9)\geq 4.32$ \cite{baumert}, $\mathscr C(C_{11})\geq 5.2895$ \cite{baumert}, $\mathscr C(C_{13})\geq 6.2743$ \cite{bohman2013}, we get new lower bounds on the Shannon capacity and can conclude $\mathcal H(J(n,3))=n$ for $n=14,18,22,26$. A better lower bound for $\mathscr C(C_{2k+1})$ would also improve the lower bound of $\mathscr C(J(4k+2,3)$.
Also, Theorem 1.1 in \cite{bohman} states that $\lim_{n\rightarrow \infty} n+\frac{1}{2}-\mathscr{C}(C_{2n+1})=0$, so we conclude $\lim_{n\rightarrow \infty} 4n+2-\mathscr{C}(J(4n+2,3))=0$.

\end{example}

The approach discussed thus far can be weak when $|\mathcal{T}_v|$ varies significantly for different vertices of $G$. In particular, it does not provide any useful result if $|\mathcal{T}_v|=0$ for only one vertex $v\in G$. To address this issue, we propose a randomized version of the above approach. Consider an arbitrary probability distribution $p(f)$ over $f\in \mathcal{F}(H,G)$. Define
$$\Gamma_1(G,H)=\min\max_{v\in\mathcal V(G)} \frac{1}{\mathbb{E}[|F(v)|]}.$$
where the minimum is over all distributions $p(f)$ on $\mathcal{F}(H, G)$ of the random function $F(\cdot)$. We have the following theorem:
\begin{theorem}
    The function $\Gamma_1(G,H)$ satisfies
\eqref{eqnv31d}.
\end{theorem}
\begin{proof}
Let $\mathcal{S}$ be a maximum independent set of $G$. For any $f \in \mathcal{F}(H, G)$, $\bigcup_{v \in \mathcal{S}} f(v)$ is an independent set of size $\sum_{v \in \mathcal{S}} |f(v)|$. Denote this independent set by $\mathcal{T}$. For a random function $F \in \mathcal{F}(H, G)$ we have,
\begin{equation*}
    \mathbb{E}[|\mathcal{T}|] = \mathbb{E}\left[\sum_{v \in \mathcal{S}} |F(v)|\right] = \sum_{v \in \mathcal{S}} \mathbb{E}[|F(v)|] \geq \left(\min_{v\in\mathcal V(G)} \mathbb{E}[|F(v)|]\right) \cdot \alpha(G).
\end{equation*}
Consequently, $\alpha(H) \geq (\min_{v\in\mathcal V(G)} \mathbb{E}[|F(v)|]) \cdot \alpha(G)$ and therefore, $$\frac{\alpha(G)}{\alpha(H)} \leq \max_{v\in\mathcal V(G)} \frac{1}{\mathbb{E}[|F(v)|]}$$

More generally, since any $f \in \mathcal{F}(H, G)$ can be transformed to $h \in \mathcal{F}(H^{\boxtimes n}, G^{\boxtimes n})$ for any choice of $n$, where $h(v_1, v_2, \cdots, v_n) = (f(v_1), f(v_2), \cdots, f(v_n))$, we have,
$$\frac{\alpha(G^{\boxtimes n})}{\alpha(H^{\boxtimes n})} \leq \left(\max_{v\in\mathcal V(G)} \frac{1}{\mathbb{E}[|F(v)|]}\right)^n.$$
Therefore, \eqref{eqnv31d} is satisfied for any distribution over $\mathcal{F}(H, G)$, including the one that minimizes $\max_{v\in\mathcal V(G)} \frac{1}{\mathbb{E}[|F(v)|]}$ so $\Gamma_1(G,H)$ satisfies
\eqref{eqnv31d}.
\end{proof}

\begin{example}\label{examassig2}
In Example \ref{examassig1}, we showed that
$$\Gamma_0(C_{2k+1},J(4k+2,3))=\frac{1}{4}.$$
Since
$$\frac{1}{4}=\frac{\alpha(C_{2k+1})}{\alpha(J(4k+2,3))}\leq
\Gamma_1(C_{2k+1},J(4k+2,3))\leq \Gamma_0(C_{2k+1},J(4k+2,3))=\frac{1}{4}$$
we deduce that
$$\Gamma_1(C_{2k+1},J(4k+2,3))=\Gamma_0(C_{2k+1},J(4k+2,3))=\frac{1}{4}.$$
    Below, we provide an example for which $\Gamma_0(G,H)>\Gamma_1(G,H)$. Let $G = C_{2k} $ and $H = J(4k-1, 3) $. We claim that for $k>3$
$$\Gamma_0(G,H)=\frac13>\Gamma_1(G,H)=\frac{k}{4k-3}.$$

We first show that $$\Gamma_1(G,H)=\frac{k}{4k-3}.$$
Since $\alpha(C_{2k})=k$ and $\alpha(J(4k-1,3))=4k-3$ \cite{nagy1972certain}, we obtain
$$\frac{k}{4k-3}=\frac{\alpha(G)}{\alpha(H)}\le \Gamma_1(G,H).$$

Let $\mathcal{V}=\{1,2,\cdots, 2k\}$. Let $f\in \mathcal F(G,H)$ be an assignment. Since $G$ is a cyclic graph, we can consider a random  $F\in \mathcal F(G,H)$ by randomly shifting  $f$ as follows:
$$F(v)=f(v+T)$$
where $T$ is uniform shift over $\mathcal{V}$ (modulo $2k$). In this case, 
$$\mathbb{E}[|F(v)|]=\frac{\sum_{v'\in G} |f(v')|}{2k}, \qquad\forall v.$$
Consequently, we obtain the following upper bound on $\Gamma_1(G,H)$:
$$\Gamma_1(G,H)\le \frac{2k}{\sum_{v\in G} |f(v)|}.$$
Consider the following function $f$: For each $i \leq 2k-2$ , let $j = 2i - (i \mod 2)$  and let $f(i)$  be all the 3-subsets of $\{ j, j+1, j+2, j+3 \}$. Also, let $f(2k-1) = \{\{ 4k-3, 4k-2, 4k-1 \}\} $ and $f(2k) = \{\{ 1, 2, 3 \}\} $. This gives a valid assignment and $\sum_i |f(v_i)| = 4(2k-2)+1+1=2(4k-3)$.

Next, we show that $$\Gamma_0(G,H)=\frac13.$$
Note that the above assignment $f$ shows that $\Gamma_0(G,H)\le \frac{1}{3}$.
Since $\Gamma_0(G,H)\geq \Gamma_1(G,H)$ and $1/\Gamma_0(G,H)$ is an integer, we obtain that
$$\Gamma_0(G,H)\geq \frac{1}{\lfloor\frac{1}{\Gamma_1(G,H)}\rfloor}=\frac13.$$
Thus, $$\Gamma_0(G,H)=\frac13.$$
\end{example}

\subsection{The second approach}
Motivated by the auxiliary receiver approach \cite{ITpaper}, we consider two graphs, $G$ and $H$, and write a bound on the ratio of their capacities as follows:
\begin{align}\frac{\mathscr{C}(G)}{\mathscr{C}(H)}&= \lim_{n \rightarrow \infty}  \left(\frac{\alpha(G^{\boxtimes n})}{\alpha(H^{\boxtimes n})}\right)^{\frac{1}{n}}\label{eqnDefC2}
\\&
= \lim_{n \rightarrow \infty}  \left(\prod_{i=1}^n\frac{\alpha(G^i\boxtimes H^{n-i})}{\alpha(G^{i-1}\boxtimes H^{n-i+1})}\right)^{\frac{1}{n}}\nonumber
\\&
= \lim_{n \rightarrow \infty}  \left(\prod_{i=1}^n\frac{\alpha(G\boxtimes W_i)}{\alpha(H\boxtimes W_i)}\right)^{\frac{1}{n}}\label{eqnDefC3A}\\&\leq
\sup_{W}\frac{\alpha(G\boxtimes W)}{\alpha(H\boxtimes W)}
\label{eqnDefC4}
\end{align}
where in \eqref{eqnDefC3A} we set $W_i=G^{i-1}\boxtimes H^{n-i}$ and the supremum in \eqref{eqnDefC4} is taken over all graphs $W$.

Define
\begin{align}
   \alpha^*(G|H)=\sup_{W}\frac{\alpha(G\boxtimes W)}{\alpha(H\boxtimes W)}, \label{EqnDef}
\end{align}
where the supremum is over \emph{all}  graphs $W$.% We will show later that this supremum is a maximum. 

By considering the special case of $W$ having a single vertex, it is clear that 
$$\alpha^*(G|H)\geq \frac{\alpha(G)}{\alpha(H)}.$$
Consequently, to show that $\alpha^*$ satisfies
\eqref{eqnv31d}, it suffices to show that
$$\alpha^*(G|H)^n\geq \alpha^*(G^{\boxtimes n}|H^{\boxtimes n}), \qquad \forall n.$$
This claim follows from the following lemma:
\begin{lemma}For any graphs $G_1, G_2, H_1$ and $H_2$ we have:
    $$\alpha^*(G_1\boxtimes G_2|H_1\boxtimes H_2)\leq \alpha^*(G_1|H_1)\alpha^*(G_2|H_2).$$ 
\end{lemma}
\begin{proof}
We have
    \begin{align*}
\alpha^*(G_1\boxtimes G_2|H_1\boxtimes H_2)&=\sup_{W}\frac{\alpha(W\boxtimes G_1\boxtimes G_2)}{\alpha(W\boxtimes H_1\boxtimes H_2)}=\sup_{W}\frac{\alpha(W\boxtimes G_1\boxtimes G_2)}{\alpha(W\boxtimes H_1\boxtimes G_2)}\cdot \frac{\alpha(W\boxtimes H_1\boxtimes G_2)}{\alpha(W\boxtimes H_1\boxtimes H_2)}
\\&\leq
\sup_{W}\frac{\alpha((W\boxtimes G_2)\boxtimes G_1)}{\alpha((W\boxtimes G_2)\boxtimes H_1)}\cdot \sup_{W}\frac{\alpha((W\boxtimes H_1)\boxtimes G_2)}{\alpha((W\boxtimes H_1)\boxtimes H_2)}
\\&\leq \alpha^*(G_1|H_1)\alpha^*(G_2|H_2).
\end{align*}
\end{proof}

Consider the special case of graph $H$ having a single vertex. In this case, we obtain 
\begin{align}\alpha^*(G|H)=\sup_{W}\frac{\alpha(G\boxtimes W)}{\alpha(W)}.\label{eqnNNfN}\end{align}
The optimization problem in \eqref{eqnNNfN} was considered by Hales in \cite{hales1973numerical} who showed that the supremum in \eqref{eqnNNfN} is a maximum and moreover,  the value of the maximum is equal to the \emph{fractional independence number} of a graph $G$ (or \emph{Rosenfeld number}), which is defined via a linear program
 as follows:
\begin{align}\alpha^*(G):=\max \sum_{v\in \mathcal V(G)}w_v\label{eqnpor1}\end{align}
where the maximum is over all weights $w_v\geq 0$ such that $\sum_{v\in \mathcal S}w_v\leq 1$ for every clique $\mathcal S\subset \mathcal V(G)$.\footnote{{Some papers define the fractional independence number differently as the 
maximum sum of non-negative weights in $[0,1]$ that can be assigned to the vertices of a graph while ensuring that the total weight on any edge does not exceed 1. Here, we require the stronger condition that the total weight of any arbitrary clique does not exceed 1.}
} In other words, the fractional independence number of $G$ is the fractional clique number of $G^c$ which is also equal to the fractional chromatic number of $G^c$. 

Since $\alpha^*(G|H)$ is equal to $\alpha^*(G)$, the \emph{fractional independence number} of $G$ when $H$ is a one-vertex graph, we name $\alpha^*(G|H)$ the \emph{relative fractional independence number} of $G$ with respect to $H$. It turns out that $\alpha^*(G|H)$ can be also characterized in terms of a linear program that generalizes the one in \eqref{eqnpor1}. However, its form is non-trivial and more involved than the one given in \cite{hales1973numerical}.

\subsection{Equivalence of the two approaches}
In this section, we provide a characterization of $\alpha^*(G|H)$ as defined in \eqref{eqnNNfN}. Observe that the third part of this characterization implies that $\Gamma_1(G,H)=\alpha^*(G|H)$. 
\begin{theorem}\label{thm1} 
Assume that $G$ is a graph with $k$ vertices and $\mathcal V (G)=\{v_1,v_2,\cdots, v_k\}$. Then, we have
\begin{enumerate}[(i)]
\item \begin{align}\alpha^*(G|H)=\max_{\mathbf w} \sum_{i=1}^k w_i,\label{eqnNewe3q}\end{align}
where the maximization over $\mathbf w=(w_1, w_2, \cdots, w_k)$ is subject to $w_i\geq 0$ and
$\sum_{i=1}^kw_i |f(v_i)|\leq 1$
for any $f\in \mathcal{F}(H,G)$. In other words, we require $\sum_{i=1}^kw_i |\mathcal T_i|\leq 1$ for any collection of sets $\mathcal T_1, \cdots, \mathcal T_k\in \mathcal I(H)$ such that $ \mathcal T_i$ and $\mathcal T_j$ are disconnected in $H$ if there is no edge between $v_i$ and $v_j$ in $G$. 

\item $$\frac{1}{\alpha^*(G|H)}= \min_{w_i\geq 0, \sum_{i=1}^kw_i=1} \max_{f}\sum_{i=1}^k w_i|f(v_i)|,$$
where the maximum is over all  $f\in \mathcal{F}(H,G)$.
\item Given a probability distribution $p(f)$ over $f\in \mathcal{F}(H,G)$, we can consider a random $F\in \mathcal{F}(H,G)$. We have
$$\alpha^*(G|H)=\min \max_{v\in\mathcal V(G)} \frac{1}{\mathbb{E}[|F(v)|]},$$
where the minimum is over all possible distributions on the set $\mathcal{F}(H,G)$.

\item For every pair of graphs $G$ and $H$, there is some graph $W$ such that 
$$\alpha^*(G|H)=\frac{\alpha(G\boxtimes W)}{\alpha(H\boxtimes W)}.$$
In other words, the supremum in the definition of $\alpha^*(G|H)$ is a maximum.
\end{enumerate}
\end{theorem}
We give two different proofs for the above theorem. %The first proof is given in Section \ref{proofthm1} and the second proof is given in Appendix B.
%The two proofs compliment each other.
 While the first proof given in Section \ref{proofthm1} is longer, it provides insights into the form of the linear program. On the other hand, if we already know the form of the linear program, its correctness can be verified directly via the second proof presented in  Appendix B.

Let us outline the first proof.
We would like to maximize 
$\alpha(W \boxtimes G)/\alpha(W \boxtimes H)$. 
Take some arbitrary graph $W$. Let $\mathcal B$ be a maximum independent set for $W \boxtimes G$, \emph{i.e.,} $|\mathcal B|=\alpha(W \boxtimes G)$. Then, for every vertex $u$ of $W$, we form the set
$\mathcal B_{u}=\{v\in\mathcal{V}(G): (u,v)\in \mathcal B\}$.
Next, we partition the vertices of $W$ according to $\mathcal B_{u}$ (vertices with the same set $\mathcal B_{u}$ will be placed in the same class). Next, using this partition, we add some edges to the graph $W$ such that $\mathcal B$ is still a maximum independent set for $W \boxtimes G$, but $\alpha(W \boxtimes H)$ can potentially decrease. Once the edges are added, the graph $W$ will have a particular structure from which we can explicitly compute $\alpha(W \boxtimes H)$ and optimize over $W$.

 {

\section{Properties of $\alpha^*(G|H)$}
This section discusses some properties of $\alpha^*(G|H)$.% Applications of $\alpha^*(G|H)$ are discussed in Section \ref{sec:app}.

%As a sanity check, one can verify that when $H$ is a single vertex, the linear program in \eqref{eqnNewe3q} reduces to the one in \eqref{eqnpor1}.

%{
%Considering $H=C_5$ or $H=C_7$ or both, please write the explicit equations.

%Provide subgraphs of size $k$ for $k=1,2,3,...$. 
%}

%{
%Another example: considering  $H=C_7^c$. Please write the explicit equations.

%Provide subgraphs of size $k$ for $k=1,2,3,...$.
%For $G=C_7$, we write an example

%}

\subsection{Properties of $\alpha^*(G|H)$}
From the definition of $\alpha^*(G|H)$ in \eqref{EqnDef}, some immediate observations can be made: for any graph $G$, we have $\alpha^*(G|H)=\alpha^*(G)$ if $H$ is a complete graph. Also, $\alpha^*(G|G)=1$. As another example, assume that the graph $H$ is universal \cite{rosenfeld1967problem}, \emph{i.e.,} $\alpha(H\boxtimes W)=\alpha(H)\alpha(W)$ for all $W$.\footnote{Examples of universal graphs are perfect graphs. A graph is a perfect graph iff its induced subgraphs include neither an odd cycle of length greater than five nor an odd anti-cycle of length greater than five. Examples of perfect graphs are complete graphs or a cycle of even length, $C_{2k}$.} For a universal graph $H$, we have
$$\alpha^*(G|H)=\frac{1}{\alpha(H)}\sup_{W}\frac{\alpha(G\boxtimes W)}{\alpha(W)}=\frac{\alpha^*(G)}{\alpha(H)}.$$
Also, if $G=G_1\boxtimes G_2$ where $G_1$ is universal, then $\alpha^*(G|H)=\alpha(G_1)\alpha^*(G_2|H)$. And if $H=H_1\boxtimes H_2$ where $H_1$ is universal, then $\alpha^*(G|H)=\frac{\alpha^*(G|H_2)}{\alpha(H_1)}$. 
 Some properties of $\alpha^*(G|H)$ are given in the following theorem.  

\begin{theorem}\label{thm2}
For any graphs $G_1, G_2, H_1$ and $H_2$ we have:
\begin{enumerate}[(i)]
\item  $\frac{1}{\alpha^*(H|G)} \leq \frac{\alpha^*(G|W)}{\alpha^*(H|W)} \leq \alpha^*(G|H)$ for any arbitrary graph $W$. 

\item $\alpha^*(G_1+G_2|H)\leq \alpha^*(G_1|H)+\alpha^*(G_2|H)$.

\item $\alpha^*((G_1+G_2)^c|H)=\max(\alpha^*(G_1^c|H),\alpha^*(G_2^c|H)).$

\item  $ \frac{1}{\alpha^*(G|H_1)} +  \frac{1}{\alpha^*(G|H_2)} \leq  \frac{1}{\alpha^*(G|H_1+H_2)}$.  Moreover, equality holds if $G$ is a vertex-transitive graph. 

\end{enumerate}

\end{theorem}
\iffalse
The following example gives an application of the above theorem:
\begin{example}Let $C_k$ be a cycle of length $k>1$. Theorem \ref{thm2} shows that
$$\alpha^*(C_{2k+1}|C_{2k+3})\geq \frac{\alpha^*(C_{2k+1})}{\alpha^*(C_{2k+3})}=\frac{2k+1}{2k+3}.$$
On the other hand, Lemma 5 of \cite{SONNEMANN1974133} implies that 
$$\alpha^*(C_{2k+1}|C_{2k+3})\leq \frac{2k+1}{2k+3}$$
Thus, 
$$\alpha^*(C_{2k+1}|C_{2k+3})=\frac{2k+1}{2k+3}.$$

\end{example}
\fi
{
\begin{remark}\label{remark2adn}
The first part of Theorem 2 implies that for any three graphs $G$, $H$ and $T$, we have
$\alpha^*(G\boxtimes T|H\boxtimes T)\leq \alpha^*(G|H)$. 
    Next, note that the first part of Theorem 2 holds with equality if $H_1$ and $H_2$ are single vertices. However,  
    strict inequality may hold in the first part of Theorem 2 if $G_1$ and $G_2$ are single vertices as
\[\alpha^*(G_1\boxtimes G_2|H_1\boxtimes H_2)<\alpha^*(G_1|H_1)\alpha^*(G_2|H_2)\]
is equivalent with 
\[\alpha(H_1\boxtimes H_2)>\alpha(H_1)\alpha(H_2)\]
which may occur.
\end{remark}
}

A proof of the above theorem is given in Section \ref{proofthm2}.
To continue, we need a property of the fractional independence number given in the following lemma.

\begin{lemma}
\label{w}
For any arbitrary graphs $G_1, G_2, G_2, \cdots, G_r$, there is some graph $W$ such that 
\begin{align*}
    \alpha^*(G_i)&=\frac{\alpha(G_i\boxtimes W)}{\alpha(W)}, \qquad \forall i.
\end{align*}

That is, a common maximizer $W$ for the optimization problems involving $\alpha^*(G_i)$ exists.
\end{lemma}

Proof of the above lemma is given in Section \ref{prooflemm1}. 
\begin{corollary} Given two arbitrary graphs $G$ and $H$, Lemma \ref{w} implies the existence of some $W$ such that
\[\frac{\alpha^*(G)}{\alpha^*(H)}=
\frac{
\frac{\alpha(G\boxtimes W)}{\alpha(W)}
}{
\frac{\alpha(H\boxtimes W)}{\alpha( W)}
}=\frac{\alpha(G\boxtimes W)}{\alpha(H\boxtimes W)}.
\]
Thus, from the definition of $\alpha^*(G|H)$ we obtain
\begin{align}\alpha^*(G|H)\geq\frac{\alpha^*(G)}{\alpha^*(H)}.\label{eqnNNn1}\end{align}
Also, we have:
\begin{align}\alpha^*(G\boxtimes H|H)=\alpha^*(G).\label{eqnImplyH}\end{align}
\end{corollary}
\begin{proof}
One direction follows from \eqref{eqnNNn1}. The other direction follows from 
\begin{align*}
\alpha^*(G\boxtimes H|H)
&=\max_{W}\frac{\alpha(G\boxtimes (H\boxtimes W))}{\alpha(H\boxtimes W)}
\leq
\max_{W}\frac{\alpha(G\boxtimes W)}{\alpha(W)}=\alpha^*(G)
\end{align*}
where the inequality follows from the fact that we are relaxing the set of graphs of the form $H\boxtimes W$ to all graphs $W$.
\end{proof}

The inequality \eqref{eqnNNn1} motivates us to ask whether $\alpha^*(G|H)$ serves as an upper bound on the ratio of other graph invariants for $G$ and $H$. The following theorem addresses this question.

\begin{theorem}\label{thm3}
For any graphs $G$ and  $H$,
$$\alpha^*(G|H)\geq \frac{X(G)}{X(H)} \geq \frac{1}{\alpha^*(H|G)}$$
where $X(G)$ can be the independence number of $G$, the fractional independence number of $G$, the Lov\'{a}sz number of $G$ \cite{Lovasz}, Schrijver's variant of the Lov\'{a}sz number of $G$ \cite{Schrijver1979,MRR1978}  or Szegedy's variant of the Lov\'{a}sz number of $G$ \cite{Szegedy1994}. 
\end{theorem}

A proof of the above theorem is given in Section \ref{proofthm3}.
Theorem \ref{thm3} does not hold when $X(G)$ is the Haemers number of a graph $G$, $\mathcal H(G)$.  We do not know if Theorem \ref{thm3} holds when $X(G)$ is the fractional Haemers number of a graph $G$. 
In the following, we give two graphs $G$ and $H$, such that $\frac{\mathcal H(G)}{\mathcal H(H)}>\alpha^*(G|H)$. 
Let $G=C_7$ and $H=J(14,3)$.
Since for any graph $G$, $\mathcal H(G)$ is an integer number, and $ 3.2578 \leq \mathscr{C}(C_7)$ \cite{polak} and $\mathscr{C}(C_7) \leq \mathcal H(C_7)$, we conclude that $\mathcal H(C_7)\geq 4$. On the other hand, in Example \ref{examassig1} we prove that $\mathcal H(J(14,3))=14$ and $\alpha^*(C_7|J(14,3))\leq \frac{1}{4}$. This implies that $\frac{\mathcal H(C_7)}{\mathcal H(J(14,3))}\geq \frac{4}{14}>\frac{1}{4}\geq\alpha^*(C_7|J(14,3))$. So, the relative fractional independence number is not always an upper bound for the ratio of Haemers number of two given graphs.

\begin{corollary}\label{ncor2n} Using Remark \ref{remark2adn}, we conclude the following chain of inequalities for  Shannon capacity:
$$\alpha^*(G|H)\geq \sup_{T}\frac{\mathscr{C}(G\boxtimes T)}{\mathscr{C}(H\boxtimes T)} \geq \frac{\mathscr{C}(G)}{\mathscr{C}(H)}\geq \inf_{T}\frac{\mathscr{C}(G\boxtimes T)}{\mathscr{C}(H\boxtimes T)}\geq \frac{1}{\alpha^*(H|G)}.$$
Observe that for any two graphs $G$ and $H$ we have $\mathscr{C}(G\boxtimes T)\geq \mathscr{C}(G)\mathscr{C}(T)$ and the strict inequality may occur \cite{haemers1979some}. We do not know if the supremum over $T$ in
$$\sup_{T}\frac{\mathscr{C}(G\boxtimes T)}{\mathscr{C}(H\boxtimes T)}$$
may occur for some non-trivial graph $T$, and if so, for which graph pairs $G$ and $H$.

\end{corollary}

\subsection{Computation of $\alpha^*(G|H)$ for a vertex-transitive graph $G$}

{
The following result is well-known (e.g., see \cite[Lemma 6]{esperet2022colouring}):
\begin{lemma}\label{lemma3ns}
    For any graph $G$, \[
\alpha^*(G)=\chi_f(G^c)\geq \frac{|\mathcal V(G)|}{\alpha(G^c)}
\]
where $\chi_f(G)$ is the fractional chromatic number of $G$. Moreover, equality holds for a vertex-transitive $G$. In other words, for a vertex-transitive graph $G$, $W=G^c$ is a maximizer of
$$\max_{W}\frac{\alpha(G\boxtimes W)}{\alpha(W)}$$
as $\alpha(G\boxtimes G^c)=|\mathcal V(G)|$ for a vertex-transitive graph $G$ 
\cite[Corollary 1(b)]{SONNEMANN1974133}.
\end{lemma}

Observe that for any arbitrary graph $G$, using the special choice of $W=G^c$ we obtain the lower bound
\begin{align}\alpha^*(G|H)=\max_{W}\frac{\alpha(G\boxtimes W)}{\alpha(H\boxtimes W)}\geq \frac{\alpha(G\boxtimes G^c)}{\alpha(H\boxtimes G^c)}\geq \frac{|\mathcal V(G)|}{\alpha(H\boxtimes G^c)}\end{align}
where  $\alpha(G^c \boxtimes G) \geq |\mathcal{V}(G)|$ follows from the fact that the set $\{(v,v): v\in\mathcal{G}\}$ is an independent set in $G^c \boxtimes G$. We provide an extension of the Lemma \ref{lemma3ns} for a vertex-transitive graph $G$ as follows:
}

\begin{theorem}
\label{alphavertran}
For any graphs $G$ and $H$, where $\mathcal V(G)=\{v_1,\dots, v_k\}$, we have
\begin{align}\alpha^*(G|H)\geq \frac{|\mathcal V(G)|}{\alpha(G^c\boxtimes H)},\label{aseqrq}\end{align}
Moreover, the lower bound in \eqref{aseqrq} is tight (holds with equality) if $G$ is vertex-transitive. In other words, for a vertex-transitive $G$, the graph
    $W=G^{c}$ is a maximizer 
    $$\alpha^*(G|H)=\max_W\frac{\alpha(G\boxtimes W)}{\alpha(H\boxtimes W)}.$$
Furthermore, for any graph $G$ on $k$ vertices $\mathcal V(G)=\{v_1,\dots, v_k\}$, we have
$\alpha(G^c\boxtimes H)=\max(\sum_{i=1}^k  |\mathcal T_i|)$
where the maximum is over all collection of sets $\mathcal T_1, \cdots, \mathcal T_k\in \mathcal I(H)$ such that $\mathcal T_i$ and $\mathcal T_j$ are disconnected in $H$ if there is no edge between $v_i$ and $v_j$ in $G$.

%or if there is a  subgraph $R$ of $G$ (on the same vertices, i.e. we only remove edges) where $R$ is vertex-transitive, and $\alpha(G^c \boxtimes H) = \alpha(R^c \boxtimes H)$.

%Suppose that $G$ is a vertex-transitive graph with $k$ vertices. Then,
%$$\alpha^*(G|H)=k\min \frac{1}{\sum_{i=1}^k  |\mathcal T_i|}=\frac{k}{\alpha(G^c\boxtimes H)},$$

\end{theorem}

A proof of the above theorem is given in Section \ref{proofthmalphavertran}.

\begin{corollary}
    { 
If $G$ is vertex-transitive, we have
    \begin{align}\alpha^*(G|H)=\frac{|\mathcal V(G)|}{\alpha(G^c\boxtimes H)}
    =\frac{|\mathcal V(G)|}{|\mathcal V(H)|}\cdot\frac{|\mathcal V(H)|}{\alpha(G^c\boxtimes H)}\leq \frac{|\mathcal V(G)|}{|\mathcal V(H)|}\cdot\alpha^*(H^c|G^c).\label{eqnfEE2}
    \end{align}
    Equality holds in \eqref{eqnfEE2} if both $G$ and $H$ are vertex-transitive. In Theorem \ref{cayex} we compute $\alpha^*(G|H)$ for certain Cayley graphs that are vertex-transitive. The above observation implies that we can also obtain the relative fractional independence number for their compliments. If both $G$ and $H$ are vertex-transitive graphs, 
    $G^c\boxtimes H$ will also be vertex-transitive and we can write
\[
\alpha^*(G|H
)= \frac{|V(G)|}{ \alpha(G^c\boxtimes H)}
= \frac{\chi_f(G^c\boxtimes H)}{ |\mathcal{V}(H)|}
\]
where we used Theorem \ref{alphavertran} and Lemma \ref{lemma3ns}. 
    
}
\end{corollary}
{
\begin{corollary}
    If both $G$ and $H$ are vertex-transitive,
    we have
    $$\alpha^*(G\star H)\geq \alpha(G\boxtimes H)$$
where $G\star H$ is the disjunctive product (also known as the co-normal product or OR product) between two graphs $G$ and $H$, defined as follows: we connect $(u,v)\in\mathcal{V}(G)\times \mathcal{V}(H)$ and $(u',v')\in\mathcal{V}(G)\times \mathcal{V}(H)$ if $uu'\in \mathcal{E}(G)$ or $vv'\in\mathcal{E}(H)$.
\end{corollary}
This inequality may be new as we have not seen it reported elsewhere. 
\begin{proof}

Consider the inequality $\alpha^*(G^c|H) \alpha^*(H|G^c) \geq 1$ which implies that
$$\alpha(G\boxtimes H)\alpha(H^c\boxtimes G^c)\leq |\mathcal{V}(G)||\mathcal{V}(H)|.$$
Since
$$\alpha^*((G^c\boxtimes H^c)^c)=\frac{|\mathcal{V}(G)||\mathcal{V}(H)|}{\alpha(G^c\boxtimes H^c)}$$
we deduce that
$$\alpha^*((G^c\boxtimes H^c)^c)\geq \alpha(G\boxtimes H)$$
Observe that
$(G^c\boxtimes H^c)^c=G\star H$. This completes the proof.

\end{proof}

Observe that $G\boxtimes H$ is a subgraph of $G\star H$. Thus, the inequality in the above corollary immediately implies the following chain of inequalities for arbitrary vertex-transitive graphs $G$ and $H$:
$$\alpha^*(G)\alpha^*(H)=\alpha^*(G\boxtimes H)\geq \alpha^*(G\star H)\geq \alpha(G\boxtimes H)\geq \alpha(G\star H)=\alpha(G)\alpha(H).$$

\begin{remark} Theorem \ref{alphavertran} allows us to use known results to obtain $\alpha^*(G|H)$ for certain special graphs. For instance, it is known \cite[Theorem 1]{SONNEMANN1974133}\cite{hales1973numerical} that for $j\geq k$ we have
$\alpha(C_{2j+1}\boxtimes C_{2k+1})=jk+\lfloor \frac{k}{2}\rfloor$ where $C_k$ is a cycle graph of size $k$. Since a cycle graph is vertex-transitive, Theorem \ref{alphavertran} allows us to compute the relative independence number between the complement of an odd cycle with respect to another odd cycle. 
More results on the independence of the strong product of three odd cycles are reported in  \cite{VESEL19989} and can be used to compute the relative fractional independence numbers for other pairs of graphs. Similar results are available for other graphs. For instance, it is proved in \cite{bad}  that 
$$\alpha(Cay(\mathbb{Z}_n, \pm 1,\pm 2,\dots \pm k)\boxtimes Cay(\mathbb{Z}_n, \pm 1,\pm 2,\dots \pm k))=\left\lfloor \frac{\left\lfloor \frac{n}{k+1}\right\rfloor n}{k+1}\right\rfloor.$$
We also remark that some of the known results in the literature follow immediately from our results. For instance, Lemma 2 in \cite{vesz} states that $2\alpha(G\boxtimes C_{2k+1})\leq (2k+1)\alpha(G)$ where $C_{2k+1}$ is a cycle of size $2k+1$. Our proof of this fact is as follows: 
\[
\frac{\alpha(C^c_{2k+1})}{\alpha(G)}\leq
\alpha^*(C^c_{2k+1}|G)=\frac{2k+1}{\alpha(C_{2k+1}\boxtimes G)}.
\]
As another example, Theorem 6 in \cite{vesz} states that
$\chi(P_5\boxtimes C_{2k+1})>5$ where $P_5$ is the Petersen graph and where $\chi(\cdot)$ is the coloring number of a graph. Our proof of this fact is as follows:
\[
\frac{\vartheta(C^c_{2k+1})}{\vartheta(P_5)}\leq \alpha^*(C^c_{2k+1}|P_5)=\frac{\chi_f(P_5\boxtimes C_{2k+1})}{10}\leq \frac{\chi(P_5\boxtimes C_{2k+1})}{10}.
\]
where $\vartheta(G)$ is the Lov\'{a}sz number of $G$.
Note that
$\vartheta(C^c_{2k+1})>2$ 
 since using the fact that the cycle graph is vertex-transitive we have
$$
\vartheta(C^c_{2k+1})=\frac{2k+1}{\vartheta(C_{2k+1})}>2.$$
Next, note that $\vartheta(P_5)=\alpha(P_5)=4$
since $P_5$ is a Kneser graph.
Thus, we obtain the desired inequality. As another example, Proposition 2.3 in \cite{jurk} states that for positive integers $k,n,p$ such that $k<n$ we have 
$$\alpha(Cay(\mathbb{Z}_n, \pm 1,\pm 2,\dots \pm k)^{\boxtimes p})\leq \left\lfloor \frac{n}{k+1}\alpha(Cay(\mathbb{Z}_n, \pm 1,\pm 2,\dots \pm k)^{\boxtimes (p-1)})  \right\rfloor$$
where $G^{\boxtimes p}$ is the strong product of $G$ with itself $p$ times. Our proof of this fact is as follows:
\begin{align*}\frac{\alpha(Cay(\mathbb{Z}_n, \pm 1,\pm 2,\dots \pm k)^c)}{\alpha(Cay(\mathbb{Z}_n, \pm 1,\pm 2,\dots \pm k)^{\boxtimes (p-1)})}&\leq \alpha^*(Cay(\mathbb{Z}_n, \pm 1,\pm 2,\dots \pm k)^c|Cay(\mathbb{Z}_n, \pm 1,\pm 2,\dots \pm k)^{\boxtimes (p-1)})\\&=\frac{n}{\alpha(Cay(\mathbb{Z}_n, \pm 1,\pm 2,\dots \pm k)\boxtimes Cay(\mathbb{Z}_n, \pm 1,\pm 2,\dots \pm k)^{\boxtimes (p-1)})}.\end{align*}
The desired inequality follows from   $\alpha(Cay(\mathbb{Z}_n, \pm 1,\pm 2,\dots \pm k)^c)\geq k+1$.

%We can extend this for Kneser graphs.

\end{remark}
}
\begin{remark}
Assume that $G$ is a single vertex, and $H$ is an arbitrary graph. By Theorem \ref{alphavertran}, we have 
\[
\alpha^*(G|H)=\frac{1}{\alpha(H)}
\]
So, in this case, computing $\alpha^*(G|H)$ is equivalent to computing $\alpha(H)$. Therefore, since computing $\alpha(H)$ is an NP-hard problem, so is computing $\alpha^*(G|H)$. Additionally, because of the hardness of the approximation of $\alpha(H)$, no efficient upper bound for $\alpha^*(G|H)$ can be found in the general case.
\end{remark}

{
We proceed with another application of Theorem \ref{alphavertran}. Note that for any graph $G$, the independence graph $G^*$ is defined in Definition \ref{defG*}. For any vertex $v\in G^*$, we can assign a number $N_v$ which is the size of the independent set in $G$ associated with the vertex $v\in G^*$.
\begin{theorem}\label{thm5na}
    For any graph $G$, if $H$ is an induced subgraph of $G^*$ then
$$\alpha^*(H|G)\leq \left(\min_{v\in \mathcal V(H)}N_v\right)^{-1}.$$
Moreover, if $H$ is an induced subgraph of $G^*$ that is also vertex-transitive, then 
$$\alpha^*(H|G)\leq \frac{|\mathcal{V}(H)|}{\sum_{v\in\mathcal{V}(H)}N_v}.$$
\end{theorem}
\begin{example}\label{example1d} 
\iffalse For a graph $G$, we can consider the set of pairs $\{v_1,v_2\}$ where the distance of $v_1$ and $v_2$  in $G$ is two. Such sets of size two are independent sets in $G$.
Let $G=C_{n}$ for some $n\geq 8$, then we obtain
independent sets of the form $\{i,i+2\}$ for $1\leq i\leq n$. The induced subgraph in $G^*$ on these independent sets is $Cay(\mathbb{Z}_n, \pm 1,\pm 2,\pm 3)$, yielding  
$$\alpha^*(Cay(\mathbb{Z}_n, \pm 1,\pm 2,\pm 3)|C_n)\leq \frac{1}{2}.$$
On the other hand, 
$$\alpha^*(Cay(\mathbb{Z}_n, \pm 1,\pm 2,\pm 3)|C_n)\geq \frac{\alpha^*(Cay(\mathbb{Z}_n, \pm 1,\pm 2,\pm 3))}{\alpha^*(C_n)}=\frac12.$$
Thus,
$$\alpha^*(Cay(\mathbb{Z}_n, \pm 1,\pm 2,\pm 3)|C_n)=\frac{1}{2}.$$
\fi
{
Let $G=Cay(\mathbb{Z}_n, \pm 1,\pm 2,\cdots,\pm k)$ for some $n\geq 4k+3$. Consider
independent sets of the form $S_i=\{i,i+k+1\}$ for $1\leq i\leq n$. The induced subgraph in $G^*$ on these independent sets is $H=Cay(\mathbb{Z}_n, \pm 1,\pm 2,\cdots,\pm (2k+1))$, yielding  
$$\alpha^*\big(Cay(\mathbb{Z}_n, \pm 1,\pm 2,\cdots,\pm (2k+1))|Cay(\mathbb{Z}_n, \pm 1,\pm 2,\cdots,\pm k)\big)\leq \frac{1}{2}.$$
On the other hand, 
$$\alpha^*\big(Cay(\mathbb{Z}_n, \pm 1,\cdots,\pm (2k+1))|Cay(\mathbb{Z}_n, \pm 1,\cdots,\pm k)\big)\geq \frac{\alpha^*(Cay(\mathbb{Z}_n, \pm 1,\pm 2,\cdots,\pm (2k+1))}{\alpha^*(Cay(\mathbb{Z}_n, \pm 1,\pm 2,\cdots,\pm k))}=\frac12.$$
Thus,
$$\alpha^*\big(Cay(\mathbb{Z}_n, \pm 1,\pm 2,\cdots,\pm (2k+1))|Cay(\mathbb{Z}_n, \pm 1,\pm 2,\cdots,\pm k)\big)=\frac{1}{2}.$$

}
\iffalse
{\color{red} Let $G=Cay(\mathbb{Z}_n, \pm 1,\pm 2,\pm 3)$ and $H=C_n$. Then,
$$\alpha^*(G+G|H)=1$$
Is it true that $G+G\in Expand(H)$?
}
\fi
\end{example}

Theorem \ref{thm5na} immediately follows from Theorem \ref{alphavertran} and the following result of \cite{SONNEMANN1974133}:
\begin{theorem}[Corrected form of Lemma 4 in \cite{SONNEMANN1974133}]
    For any graphs $G$ and $W$ we have
\begin{align}
    \alpha(G\boxtimes W)&=\max_{H\subset G^*}\big\{\alpha(H\boxtimes W)\min_{v\in \mathcal V(H)}N_v\big\}\label{eqnnne1r}\end{align}
    where  
    $H\subset G^*$ means that $H$ is an induced subgraph of $G^*$. 
    Next, if $H$ is an induced subgraph of 
     $G^*$, then
    \[
        \alpha(G\boxtimes H^c)\geq \sum_{v\in\mathcal{V}(H)}N_v
    \]

\end{theorem}

Finally, a method for the construction of \underline{non-vertex-transitive} graphs $T$ and $H$ for which 
$$\alpha^*(T|H)=\frac{|\mathcal{V}(T)|}{\alpha(T^c \boxtimes H)}$$
holds is given in the following proposition.
\begin{proposition}
Take an arbitrary vertex-transitive graph $G$ and an arbitrary graph $H$. Let $T$ be obtained by adding some new edges to $G$. The addition of new edges to graph $G$ is done in such a way that $\alpha(T^c \boxtimes H) = \alpha(G^c \boxtimes H)$
then 
$$\alpha^*(T|H) = \alpha^*(G|H)=\frac{|\mathcal{V}(T)|}{\alpha(T^c \boxtimes H)}$$
\end{proposition}
\begin{proof}
    Since $T\in Expand(G)$, from Lemma \ref{lemmaexpand} we have $\alpha^*(G|H) \geq \alpha^*(T|H)$. To show the reverse direction, note that 
\begin{align*}
    \alpha^*(G|H)=\frac{|\mathcal{V}(G)|}{\alpha(G^c \boxtimes H)}
    =\frac{|\mathcal{V}(G)|}{\alpha(T^c \boxtimes H)}\leq \alpha^*(T|H).
\end{align*}
This shows $\alpha^*(G|H) = \alpha^*(T|H)$. 
\end{proof}

}

\subsubsection{The Expand Operations}
Finally, we provide another definition that is useful in studying $\alpha^*(G|H)$.
\begin{definition}[Expand Operations]
    \label{def:expand}
    Let $G$ be an arbitrary graph. We define $Expand(G)$ as the set of all graphs that can be obtained by a sequence of the following three operations on $G$:
    \begin{enumerate}[(I)]
        \item Remove a vertex $v$ from $G$.
        \item Replace a vertex $v$ of $G$ by a clique of arbitrary size.     In other words, put a clique of size $k$ instead of vertex $v$ and connect these $k$ new vertices to all neighbors of $v$.
        \item Add a new edge to $G$.
    \end{enumerate}
\end{definition}

\begin{remark}
\label{mrk:merge}
    Note that by the above operations, one can merge (contract) any two vertices $v$ and $u$ in $G$. To achieve this, we only need to delete one of the vertices, for instance $v$, by applying an operation of type (I). Then, we add new edges between $u$ and the neighbors of $v$ by applying an operation of type (III). This allows us to merge the two vertices $v$ and $u$ into a single vertex.
\end{remark}
\begin{remark}
    \label{rem:expand-order}
    Definition \ref{def:expand} gives three types of operations.
    To reach a graph $H$ from a given graph $G$ using the Expand operations, we can always assume that the operations occur in the following order: first a series of operations (I), then a series of operations (II), and finally a series of operations (III). To see this, observe that the operations of adding an edge (type (III)) can always be postponed to the very end. Then, if we have a sequence of operations of type (I) or (II), the operations of removing a vertex (type (I)) can all be performed first at the beginning. 
\end{remark}

    One can easily verify that all the above Expand operations, when applied to a graph $G$, will not increase $\alpha(G \boxtimes W)$ for any choice of graph $W$. Consequently, we obtain the following lemma:
    \begin{lemma}\label{lemmaexpand}
        If $G'\in Expand(G)$ and $H'\in Expand(H)$, then we have
        $$\alpha^*(G'|H)\leq \alpha^*(G|H)\leq \alpha^*(G|H').$$ 
    \end{lemma}

    \begin{corollary}\label{Cor:expandconj}
        If $G \in Expand(H)$, then $\alpha^{*}(G|H) \leq \alpha^*(H|H)= 1$.
    \end{corollary}

\begin{corollary}\label{corrnew23}
If $G \in Expand(H)$ and $\alpha(G)\geq \alpha(H)$, then $\alpha^{*}(G|H) = 1$. 
\end{corollary} 
\begin{proof}
Since $G \in Expand(H)$, we have $\alpha^{*}(G|H) \leq 1$. On the other hand, $\alpha^{*}(G|H) \geq \frac{\alpha(G)}{\alpha(H)}\geq 1$.

\end{proof}
\begin{example} 
\label{example2d}
As an example, we compute the relative fractional independence number of a certain Cayley graph with respect to another Cayley graph. Let $G = Cay(\mathbb{Z}_{(r+2)k+r}, \pm 1,\pm 2, \dots, \pm (r+1))$ and $H = Cay(\mathbb{Z}_{(r+1)k+r}, \pm 1,\pm 2, \dots, \pm r)$ for natural numbers $r$ and $k$. We show that $\alpha^*(G| H) = 1.$

Note that $$\alpha(G) = \left\lfloor \frac{(r+2)k+r}{r+2} \right\rfloor = k, \qquad \alpha(H)= \left\lfloor \frac{(r+1)k+r}{r+1} \right\rfloor = k.$$ To utilize Corollary \ref{corrnew23}, we need to show that we can derive $G$ from $H$ by using the expand operations.
A demonstration for the special case of $k=3, r = 1$ is given in Figure \ref{fignw1}. We proceed as follows: first label the $i$-th vertex of $H$ with $$u_{i+\left\lceil \frac{i}{r+1}\right\rceil - 1}$$ 
for $i=1,2,\cdots, (r+1)k+r$. In other words, the sequences of vertex labels are $$u_1, u_2, \cdots, u_{r+1}, u_{r+3}, u_{r+4}, \cdots, u_{2r+3}, u_{2r+5},u_{2r+6},...$$ with no vertex labeled as $u_{(r+2)i}$.
Now, for any $1 \leq i \leq k$, replace each vertex $u_{(r+2)i-1}$ with a clique of size two and label the two new vertices with $u_{(r+2)i-1}$ and $u_{(r+2)i}$. The resulting graph has $(r+2)k+r$ vertices. The neighbors of each vertex is as follows:
         \[N(u_i)=\begin{cases}
         \lbrace u_{i-(r+1)}, \dots ,u_{i-1}, u_{i+1}, \dots, u_{i+r} \rbrace, &  \text{if $i$ is divisible by $(r+2)$}, \\
        \lbrace u_{i-(r+1)}, \dots ,u_{i-1}, u_{i+1}, \dots, u_{i+r+1} \rbrace, &\text{if the remainder of $i$ divided by $(r+2)$ is between $1$ and $r$, }  \\
        \lbrace u_{i-r}, \dots ,u_{i-1}, u_{i+1}, \dots, u_{i+r+1} \rbrace, &\text{if the remainder of $i$ divided by $(r+2)$ is $r+1$}. 
        \end{cases}
        \]
This is a spanning subgraph of $G$. Therefore, by adding the missing edges to this graph, one can obtain $G = Cay(\mathbb{Z}_{(r+2)k+r}, \pm 1,\pm 2, \dots, \pm (r+1))$.
    \begin{figure}
        \centering
        \includegraphics[width=.3\textwidth]{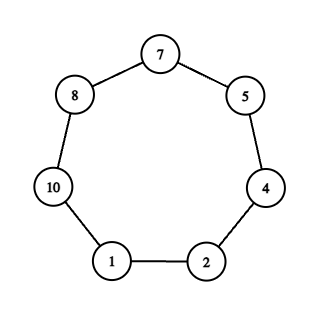}\hfill
        \includegraphics[width=.3\textwidth]{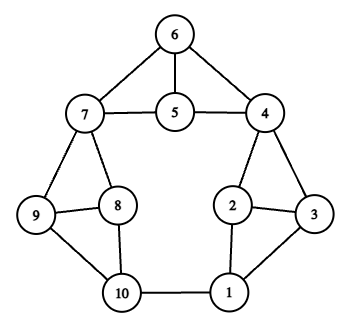}\hfill
        \includegraphics[width=.3\textwidth]{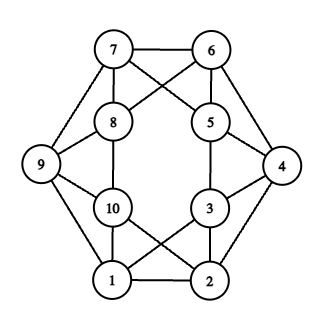}
        \caption{Replacing vertices 2, 5, and 8 with a clique of size two, and then adding extra edges (1,9), (2,10), (3,5), and (6,8) yields the graph $Cay(\mathbb{Z}_{10}, \pm 1,\pm 2)$. \label{fignw1}}
    \end{figure}

\end{example}
We continue by a remark about Corollary \ref{corrnew23}:
\begin{remark}
    Take two arbitrary graphs, $G$ and $H$ and two natural numbers $p$ and $q$. Let $$H'=\underbrace{H+H+\cdots+H}_{p \text{ times}} \, \, \text{and} \, \, G'=\underbrace{G+G+\cdots+G}_{q \text{ times}}.$$  Then, for any graph $W$, $\alpha(H' \boxtimes W) = p \cdot \alpha(H \boxtimes W)$ and $\alpha(G' \boxtimes W) = q \cdot \alpha(G \boxtimes W)$, and therefore $\alpha(G|H) \leq \frac{p}{q}$ is equivalent to $\alpha(G'|H') \leq 1$. Therefore, one may potentially utilize Corollary \ref{corrnew23} to prove $\alpha(G|H) = \frac{p}{q}$.
\end{remark}

{
Next, we have the following result:
\begin{theorem}\label{thmN4new}
    Take arbitrary graphs $G$ and $H$, and let $\mathcal{S}$ be an independent set of $G^c\boxtimes H$. Define $\mathcal{S}_G\subset\mathcal{V}(G)$ as the set of vertices of $G$ used in the independent set as follows:
    $$\mathcal{S}_G=\{u\in \mathcal{V}(G): (u,v)\in \mathcal S \text{ for some } v\in \mathcal V(H)\}.$$
    Let $G'$ be the induced subgraph of $G$ on vertices $\mathcal{S}_G$. Similarly, we define $H'$ as the induced graph on the vertices of $H$ used in the independent set $\mathcal{S}$.
    Then, $G'\in Expand(H')$ and consequently, $\alpha^*(G'|H')\leq 1$. In particular, since $H'\in Expand(H)$, we also obtain $G'\in Expand(H)$ and $\alpha^*(G'|H)\leq 1$.
\end{theorem}
A proof of the above theorem is given in Section \ref{proofthmN4new}.

\subsubsection{$G \in Expand(H)$ vs $\alpha^{*}(G|H)
\leq 1$}
From Lemma \ref{lemmaexpand}, it follows that
if $G \in Expand(H)$, then $\alpha^{*}(G|H) \leq \alpha^{*}(H|H)=1$. One might ask if the converse is also true, i.e., whether $\alpha^{*}(G|H)\leq 1$ implies that $G \in Expand(H)$. 
We show that the converse statement is false in general, but holds for a certain class of Cayley graphs (which includes all cycles) and for all perfect graphs.

\begin{theorem}\label{Thm:bellowcayley} In the following cases
$\alpha^*(G|H) \leq 1$ implies $G \in Expand(H)$:
\begin{enumerate}[i.]
        \item If $G$ is the Cayley graph $Cay(\mathbb{Z}_n, \pm 1,\pm 2,\dots \pm k)$, where $n = c(k+1)+1$ or $n = c(k+1)$ for an arbitrary integer $c$, and $H$ is an arbitrary graph
    \item If $G$ is a perfect graph and $H$ is an arbitrary graph.
\end{enumerate}
\end{theorem}

A proof of the above theorem is given in Section \ref{proofthmCayley}.  Note that by setting $k = 1$ in the first part of Theorem \ref{Thm:bellowcayley}, we observe that Theorem \ref{Thm:bellowcayley} holds when $G$ is a cycle.

Next, we construct an example for which $\alpha^*(G|H) \leq 1$ but $G \notin Expand(H)$. Set $G=Cay(\mathbb{Z}_n, \pm 1,\pm 2,\pm 3)$ and $H=C_n$ where $n$ is an arbitrary odd number greater than $7$. Then, Example \ref{example1d} yields $\alpha^*(G|C_n)=\frac{1}{2}$ which is equivalent to $$\alpha^*(G + G|C_n)=1.$$

On the other hand, $G+G \notin Expand(C_n)$ follows from Lemma \ref{lemmaNex} below as for any odd $n$ we have
$$\alpha^*(G+G)=2\alpha^*(G)=\frac{n}{2}>\left\lfloor \frac n2\right\rfloor.$$
\begin{lemma}\label{lemmaNex}
    If $T$ is a disconnected graph and $T\in Expand(C_n)$, then
$$\alpha^*(T)\leq \alpha(C_n)=\left\lfloor \frac n2\right\rfloor$$
\end{lemma}
\begin{proof}
    As argued in Remark \ref{rem:expand-order}, one can argue if $T=Expand(H)$, then one can also reach $T$ from $H$ by first performing some operations of type (I), then some operations of type (II), and finally some operations of type (III). Since $T$ is a disconnected graph while $C_n$ is a connected graph, we must utilize the vertex deletion operation (operation of type (I)) in the first stage. This is because the second and third Expand operations alone cannot decrease the connectivity of a graph. Let $H'$ be the graph obtained from removing at least one vertex from $H=C_n$. Then $T$ will be in the Expand of $H'$, and therefore $\alpha^*(T|H')\leq 1$. Note, however, that the graph $H'$  will be a perfect graph (and thus a universal graph); removal of any one vertex from $C_n$ will make it a perfect graph. Therefore,
$\alpha^*(T|H')=\alpha^*(T)/\alpha(H')$. Thus,
$$\alpha^*(T)\leq \alpha(H')\leq \alpha(C_n)=\left\lfloor \frac n2\right\rfloor.$$
\end{proof}

Finding necessary or sufficient conditions for $G=Expand(H)$ is an interesting open problem. As an example, we give such a condition below. Let $\omega(G)=\alpha(G^c)$ be the clique number of $G$ (size of the maximum clique in $G$). 
\begin{proposition}
    If $G\in Expand(H)$, then one can assign a non-negative \emph{integer} $s(v)$ to every vertex in $H$ such that
    $|\mathcal{V}(G)|=\sum_{v\in \mathcal{V}(H)}s(v).$
    Moreover, for every clique $C$ in $H$ we have
$\sum_{v\in C}s(v)\leq \omega(G).$
\end{proposition}
\begin{proof}
    Note that operation (III) in the expand operations can only increase the size of the maximum clique. Thus,
after performing operations of type (I) and (II) on $H$, we cannot reach a graph whose  clique number is larger  $\omega(G)$. 
Assume that in operations (I) and (II) we replace every vertex $v$ of a graph $H$ with $s(v)$ copies ($s(v)=0$ is equivalent to removing the vertex, and $s(v)=1$ is equivalent with keeping the vertex untouched while $s(v)>1$ is equivalent to replicating the vertex). Then, for every clique $C$ in $H$ we must have
$\sum_{v\in C}s(v)\leq \omega(G).$
On the other hand, we must have $|\mathcal{V}(G)|=\sum_{v\in \mathcal{V}(H)}s(v)$ as $\sum_{v\in \mathcal{V}(H)}s(v)$ is the total number of vertices after performing operations of type (I) and (II).
\end{proof}

\begin{lemma}
    We have $G\in Expand(H)$ if and only if there exists some $f \in \mathcal{F}(H, G)$ such that $|f(v)| \geq 1$ for all  $v \in \mathcal{V}(G)$.
\end{lemma}
\begin{proof}
    Assume that $G\in Expand(H)$. By Remark \ref{rem:expand-order}, every vertex $v$ in $G$ is constructed by replacing some vertex $u \in \mathcal{V}(H)$ with a clique (the clique is of size one when we just keep the vertex in $H$ and not replicate it). We now produce a function  $f \in \mathcal{F}(H, G)$ as follows: if a vertex $v \in \mathcal{V}(G)$ is constructed by replacing $u \in \mathcal{V}(H)$ with a clique, we set $f(v) = \{ u\}$. The function $f$ created in this way satisfies the conditions of $\mathcal{F}(H, G)$. To see this, assume that $v, w \in \mathcal{V}(G)$ are disconnected, then they cannot be obtained from two connected vertices in $H$ as we do not delete edges during the expand operations.

Conversely, assume that there exists some $f \in \mathcal{F}(H, G)$, such that $|f(v)| \geq 1$ for all  $v \in \mathcal{V}(G)$. For any $v \in \mathcal{V}(G)$, arbitrarily delete vertices from $f(v)$ until only one remains. Now, to obtain $G$ from $H$ using the Expand operations, we first remove vertices of $H$ that do not appear in $f(v)$ for all $v$. Then, we  substitute each vertex $u \in \mathcal{V}(H)$ with a clique of size $|f^{-1}(u)|$. Adding the remaining edges would construct $G$ from $H$. This can be done because the sets $|f^{-1}(u)|$ construct a partition of the vertices in $G$.
\end{proof}

Recall the definition of $\Gamma_0(G,H)=\min_{f} \max_{v\in\mathcal V(G)} \frac{1}{|f(v)|}$. By the above lemma, $G \in Expand(H)$ if and only if $\Gamma_0(G,H) \leq 1$. Thus, we observe that for all graphs satisfying the assumptions of Theorem \ref{Thm:bellowcayley}, including cycles and perfect graphs, $\alpha^*(G|H)=\Gamma_1(G,H) \leq 1$ if and only if $\Gamma_0(G,H) \leq 1$. However, Example \ref{example1d} provides an example in which $\Gamma_0(G,H) =\infty$ while $\alpha^*(G|H) \leq 1$.

\section{Applications of the relative fractional independence number}

\label{sec:app}

In this section, we explore some applications of $\alpha^*(G|H)$.

\subsection{Bounds on the independence number}
\label{app3ttt}
In \cite{bohman2}, the structure of odd cycles is used to show that  for any natural number $d$, $m\geq 3$ and even number $\beta>0$,  \begin{align}\alpha(C^d_{m+\beta})\geq \alpha(C^d_{m}){\left(\dfrac{m+\beta}{m}\right)}^d\label{eqnlast}
\end{align}
where $C^d_{m}$ is the $d$-th power of $C_m$. One can also obtain this result directly using the No-Homomorphism lemma. In Appendix A, we prove that $\alpha^*(C_n|C_m)=n/m$ for odd $m$ and $n$ satisfying $n<m$. Thus, we recover the result of \cite{bohman2}. More generally, observe that for any natural number $k$, 
$$(n/m)^k=\frac{\alpha^*(C_n^{\boxtimes k})}{\alpha^*(C_m^{\boxtimes k})}\leq \alpha^*(C_n^{\boxtimes k}|C_m^{\boxtimes k})\leq (\alpha^*(C_n|C_m))^k=(n/m)^k.$$  Thus, for odd $m$ and $n$ satisfying $n<m$, we deduce $$\alpha^*(C_n^{\boxtimes k}|C_m^{\boxtimes k})=(n/m)^k, \qquad \forall k.$$
So, we conclude that $n/m$ is the fundamental limit of the ratio one can obtain by the technique of assigning independent sets of $C_m^{\boxtimes k}$ to vertices of $C_n^{\boxtimes k}$ for any arbitrary $k$. 

In the next sub-section, we discuss the more general class of Cayley graphs and compute $\alpha^*(G|H)$. This provides a different generalization of \eqref{eqnlast}.

By Theorem \ref{thm3}, the inequality $$\frac{\mathscr{C}(G)}{\mathscr{C}(H)}\leq \alpha^*(G|H)$$ yields a non-trivial upper bound on the ratio of Shannon number of $G$
and $H$. One can compare this bound by the  individual bounds on $\mathscr{C}(G)$ and $\mathscr{C}(H)$. In particular, we can write
\begin{align}\frac{\mathscr{C}(G)}{\mathscr{C}(H)}\leq \frac{\mathscr{C}(G)}{(\alpha(H^d))^{1/d}}\label{eqnReT}
\end{align}
for any natural number $d$ and then use a known upper bound on $\mathscr{C}(G)$ (such as the Lov\'{a}sz number) to compute an upper bound on the ratio of the capacities of $G$ and $H$. However, calculating $\alpha(H^d)$ for large values of $d$ is computationally difficult, and choosing small values of $d$ can lead to weak bounds. As an example, for a small graph like $C_7$, it is known that $\alpha(C^3_7)=33$, and for the larger values of $d$, we only know the bounds $108\leq\alpha(C^4_7)\leq 115$ and $367\leq\alpha(C^5_7)\leq 401$ \cite{polak}. Since $\frac{\mathscr{C}(C_9)}{\mathscr{C}(C_{11})}\leq \alpha^*(C_9|C_{11})=\frac{9}{11}$, while the best we can have using the ratio of capacities is that $\mathscr{C}(C_9)\leq 4.3601$ (Lovasz bound) and $\mathscr{C}(C_{11})\geq 148^{\frac{1}{3}} > 5.2895$ \cite{baumert}, so $\frac{\mathscr{C}(C_9)}{\mathscr{C}(C_{11})}\leq \frac{4.3601}{5.2895}$. We have $\frac{9}{11}\leq \frac{4.3601}{5.2895}$.

\subsection{The ratio of Shannon capacities of two Cayley graphs}
\label{sec:newsec}

%The inequality $\dfrac{\mathscr C (G)}{\mathscr C(H)}\leq \alpha^*(G|H)$ 

As a concrete example, we compute $\alpha^*(G|H)$ for two Cayley graphs obtained by cyclic groups. 

\begin{theorem}\label{cayex}
Let $G=Cay(\mathbb{Z}_n, \pm 1,\pm 2,\dots \pm k)$ and $H=Cay(\mathbb{Z}_m, \pm 1,\pm 2,\dots, \pm k)$, where $1\le 2k<n<m$ are integers. Then, $\alpha^*(G|H)\ge \frac{n}{m}$. Moreover, $\alpha^*(G|H)=\frac{n}{m}$ if there are integers $\ell,s\ge 0$ such that
$m=\ell n+s(k+1)$.
\end{theorem}
Proof of the above theorem is given in Section \ref{proofcayex}. It utilizes the idea of defining a homomorphism from $H$ to $G$. Theorem \ref{cayex} implies $$\frac{\mathscr C (Cay(\mathbb{Z}_n, \pm 1,\pm 2,\dots \pm k))}{\mathscr C ( Cay(\mathbb{Z}_m, \pm 1,\pm 2,\dots \pm k))}\leq \frac{n}{m}$$ for $m=\ell n+s(k+1)$ for some integers $\ell,s\ge 0$. This inequality is an explicit upper bound on the ratio of the capacities of Cayley graphs. We believe this result is novel. 
To compare this bound with the previously known bounds, one can use individual bounds on $\mathscr{C}(G)$ and $\mathscr{C}(H)$ to write the following upper bound.
In \cite{bad}, it is proved that 
$$\alpha(Cay(\mathbb{Z}_n, \pm 1,\pm 2,\dots \pm k)\boxtimes Cay(\mathbb{Z}_n, \pm 1,\pm 2,\dots \pm k))=\left\lfloor \frac{\left\lfloor \frac{n}{k+1}\right\rfloor n}{k+1}\right\rfloor.$$
So, $$\frac{\mathscr C (Cay(\mathbb{Z}_n, \pm 1,\pm 2,\dots k))}{\mathscr C ( Cay(\mathbb{Z}_m, \pm 1,\pm 2,\dots k))}
\leq \frac{\alpha^* (Cay(\mathbb{Z}_n, \pm 1,\pm 2,\dots \pm k))}{\alpha(Cay(\mathbb{Z}_m, \pm 1,\pm 2,\dots \pm k)^2)}=\frac{n/(k+1)}{\sqrt{\left\lfloor \frac{\left\lfloor \frac{m}{k+1}\right\rfloor m}{k+1}\right\rfloor}},$$ which is strictly greater than $\frac{n}{m}$ if $m$ is not divisible by $k+1$.

\subsection{Homomorphism and the relative fractional independence number}
\label{nohomosection}

%A homomorphism $\mathsf{g}:H\rightarrow G$ from a graph $H$ to a graph $G$ is a map $\mathsf{g}:\mathcal V(H)\rightarrow \mathcal V(G)$ such that $uv\in \mathcal E(H)$ implies $\mathsf{g}(u)\mathsf{g}(v)\in \mathcal E(G)$.

The well-known No-Homomorphism Lemma states that\footnote{See  \cite[ Exercise 2.12]{hammack} for an extension of the No-Homomorphism Lemma based on the maximum number of vertices in an induced sub-graph of $G$ that is homomorphic to an auxiliary graph $K$.}:
\begin{lemma}\cite{Albertson}
\label{nohomo}
 If there is a homomorphism from $H$ to $G$, and $G$ is vertex-transitive, then $\frac{\alpha(G)}{\alpha(H)}\leq \frac{|\mathcal V(G)|}{|\mathcal V(H)|} $.
\end{lemma}
\iffalse
\begin{lemma}\cite[ Exercise 2.12]{hammack}
\label{nohomoext}
For any graph $K$, let $n(G,K)$ be the maximum number of vertices in an induced sub-graph of $G$ that is homomorphic to $K$ (i.e., there exists at least one homomorphism between that induced sub-graph and $K$), Given graphs $G,H$ and, $K$, we have 
$$\frac{n(G,K)}{n(H,K)}\leq \frac{|\mathcal V(G)|}{|\mathcal V(H)|},$$ if $G$ is vertex-transitive and there is a homomorphism from $G$ to $H$.
\end{lemma}
If in Lemma \ref{nohomoext}, we set $K$ as a graph with one vertex, then we get the No-Homomorphism Lemma.
\fi
We can use No-Homomorphism Lemma to show the non-existence of homomorphism from $H$ to $G$.
Also, if we show that there is a homomorphism from $H$ to $G$,
then we can use No-Homomorphism Lemma to give an upper bound on the $\alpha(G)$ (See \cite{albertson1}).
We give an extension of the No-Homomorphism Lemma using the relative fractional independence number.
\begin{lemma}\label{hn}
    If $G$ and $H$ are two graphs where there is a homomorphism $\mathsf{g}:H\rightarrow G$ then $\alpha(G^c\boxtimes H)\ge |\mathcal V(H)|$.
\end{lemma}
\begin{proof}
    Note that the collection of vertices $(\mathsf{g}(u),u)$ in $G^c\boxtimes H$ for all $u\in \mathcal V(H)$ is an independent set of size $|\mathcal V(H)|$.
\end{proof}

\begin{lemma}
\label{ournohomo}
If there is a homomorphism from $H$ to $G$, and $G$ is vertex-transitive, then $\alpha^*(G|H)\leq \frac{|\mathcal V(G)|}{|\mathcal V(H)|} $.
\end{lemma}
\begin{proof}
Since $G$ is vertex-transitive, by Theorem \ref{alphavertran},
$ \alpha^*(G|H)=\frac{|\mathcal V(G)|}{\alpha(G^c\boxtimes H)}$.
By Lemma \ref{hn}, since there is a homomorphism from $H$ to $G$,
we have $\alpha(G^c\boxtimes H)\ge |\mathcal V(H)|$.

\end{proof}

We have the following corollary by Lemma \ref{ournohomo} and Theorem \ref{thm3}.
\begin{corollary}
    If there is a homomorphism from $H$ to $G$, and $G$ is vertex-transitive, then $\frac{X(G)}{X(H)}\leq \frac{|\mathcal V(G)|}{|\mathcal V(H)|} $, where $X(G)$ can be the  Shannon
capacity of $G$, the fractional independence number of $G$, the Lov\'{a}sz number of $G$, Schrijver's
or Szegedy's variants of the Lov\'{a}sz number.

\end{corollary}
The inequality 
$$\frac{\alpha^*(G)}{\alpha^*(H)}\leq \frac{|\mathcal V(G)|}{|\mathcal V(H)|}$$
(and similarly for the Lov\'{a}sz number) is novel and cannot be directly derived from the original No-Homomorphism Lemma.

\section{Some open problems}
\label{sec:futurework}

\begin{itemize}
    \item 
    We derived the following inequalities:
    $$\frac{1}{\alpha^*(H|G)} \leq \frac{\alpha^*(G|W)}{\alpha^*(H|W)} \leq \alpha^*(G|H) \qquad \forall \text{ graph }W,$$
    and
    $$\frac{1}{\alpha^*(H|G)}\leq \frac{X(G)}{X(H)} \leq 
    \alpha^*(G|H)
    $$
    where $X(\cdot)$ can be any of the properties described in Theorem \ref{thm3}. One might be interested in characterizing graphs $G$ and $H$ for which $\alpha^*(G|H) \alpha^*(H|G) = 1$ holds, implying that all of the above inequalities hold with equality. 
The equality
$\alpha^*(G|H) \alpha^*(H|G) = 1$
holds if and only if
$$\frac{\alpha(G\boxtimes W)}{\alpha(H\boxtimes W)}=\frac{\alpha(G)}{\alpha(H)}$$
for every graph  $W$. In other words, every graph $W$ should be a maximizer of $\alpha^*(G|H)$. An example that this may occur is as follows: suppose that
$$G = G_1 \boxtimes G', \qquad H = G_1 \boxtimes H'$$
where $G_1$ is some arbitrary graph, and both $G'$ and $H'$ are universal graphs. Then the equality $\alpha^*(G|H) \alpha^*(H|G) = 1$ holds. However, we do not have a full characterization of when the equality $\alpha^*(G|H) \alpha^*(H|G) = 1$ holds in general.

\item Is there any  connections between $\alpha^*(G|H)$ and the fractional Haemers numbers $G$ and $H$? 

\item Is there an efficient algorithm to check whether $G \in Expand(H)$? 

\item 
In this paper, we considered the following expansions:
\begin{align}\alpha(G^{\boxtimes n})=\prod_{i=1}^n 
\frac{\alpha(G^{\boxtimes i})}{\alpha(G^{\boxtimes (i-1)})}
=\prod_{i=1}^n 
\frac{\alpha(G\boxtimes G^{\boxtimes (i-1)})}{\alpha(G^{\boxtimes(i-1)})}
\leq \left[\sup_{W}\frac{\alpha(G\boxtimes W)}{\alpha(W)}\right]^n,\label{eqnRep}
\end{align}
and
\begin{align}\frac{\alpha(G^{\boxtimes n})}{\alpha(H^{\boxtimes n})}
= \prod_{i=1}^n\frac{\alpha(G\boxtimes G^{\boxtimes{i-1}}\boxtimes H^{\boxtimes{n-i}})}{\alpha(H\boxtimes G^{\boxtimes{i-1}}\boxtimes H^{\boxtimes{n-i}})}\leq 
\left[\sup_{W}\frac{\alpha(G\boxtimes W)}{\alpha(H\boxtimes W)}\right]^n.
\label{eqnRep2}
\end{align}
Let $G=C_5$ and $F=C_5^{\boxtimes 2}$.
It is conjectured that
for $C_5$ we have (see  \cite{SONNEMANN1974133})
$$\alpha(G^{\boxtimes n})=\begin{cases}
    5^{\frac{n}{2}}& \text{ if $n$ is even},\\
    2\times 5^{\frac{n-1}{2}}& \text{ if $n$ is odd}.
\end{cases}$$
Consequently, the  limit $\lim_{i\rightarrow\infty}\dfrac{\alpha(G\boxtimes G^{i-1})}{\alpha(G^{i-1})}$
may not exist for $G=C_5$, but the limit certainly exists for $F=C_5^{\boxtimes 2}$. Considering  \eqref{eqnRep} for $F$, the following upper bound is used:
\begin{align}\dfrac{\alpha(F\boxtimes F^{i-1})}{\alpha(F^{i-1})}\leq \sup_{W}\frac{\alpha(F\boxtimes W)}{\alpha(W)}.\label{eqnrrkj3}\end{align}
Observe that for every $i$,
$$\dfrac{\alpha(F\boxtimes F^{i-1})}{\alpha(F^{i-1})}=5, \qquad  \sup_{W}\frac{\alpha(F\boxtimes W)}{\alpha(W)}=\alpha^*(F)=6.25.$$
We obtain a weak upper bound in \eqref{eqnrrkj3} 
since we relax the supremum over all possible graphs $W$. Thus, the upper bounds in
\eqref{eqnRep} and \eqref{eqnRep2} may be improved if we can restrict the set of graphs $W$ that we maximize over.

%This motivates  the following conjecture. Proving or disproving the conjecture helps understand or improve the bounds in \eqref{eqnRep} and \eqref{eqnRep2}.
%\begin{conjecture}The  limit $\lim_{i\rightarrow\infty}\dfrac{\alpha(G\boxtimes G^{i-1})}{\alpha(G^{i-1})}$
%exists for any graph $G$. Moreover, for any $c\in(0,1)$, the limit
%\[\lim_{n\rightarrow\infty}
%\frac{\alpha(G\boxtimes G^{nc}\boxtimes H^{n(1-c)})}{\alpha(H\boxtimes G^{nc}\boxtimes H^{n(1-c)})}
%\]
%exists for any graph $G$ and $H$. 
%\end{conjecture}
\end{itemize}

%The inequality \eqref{eqnCCC3} might be of independent interest: for instance, let $T=H^c$ and assume that $H$ is vertex-transitive. Since $\mathscr{C}(H\boxtimes H^c)=|\mathcal V(H)|$ \cite[Theorem 12]{lovasz1979shannon}
%we obtain the following inequality for any arbitrary $G$ and vertex-transitive $H$:
%\[\alpha^*(G|H)\geq
%\frac{\mathscr{C}(G\boxtimes H^c)}{|\mathcal V(H)|}.\]

\section{Proofs}
\label{proofsection}
\subsection{Proof of Theorem \ref{thm1}}\label{proofthm1}
Observe that the theorem's second and third parts follow from the first part. For the third part, the dual of the linear program in the first part is as follows: assigning a non-negative weight $\beta_f$ to $f\in \mathcal{F}(H,G)$ and multiplying the equation $\sum_{i=1}^kw_i |f(v_i)|\leq 1$ by $\beta_f$; we obtain the following dual to the linear program of the first part: $\alpha^*(G|H)$ equals the minimum of $\sum_{f}\beta_f$ such that 
$$\sum_{f} \beta_f|f(v)|\geq 1, \qquad \forall v\in \mathcal V(G).$$
Equivalently, $\alpha^*(G|H)$ equals the minimum over $\beta_f$ of
$$\frac{\sum_{f}\beta_f}{\min_{v\in \mathcal V(G)}\sum_{f} \beta_f|f(v)|}.$$
By scaling $\beta_f$, we can impose the constraint $\sum_{f}\beta_f=1$ and view $\beta_f$ as a probability distribution over $f\in \mathcal{F}(H,G)$. Hence, the third part of the theorem follows. The third part implies the second part, as the second part implies that 
$$\frac{1}{\alpha^*(G|H)}=\max \min_{v\in\mathcal V(G)} \mathbb{E}[|F(v)|].$$
Using the minimax theorem and exchanging the order of minimum and maximum, we have
$$\frac{1}{\alpha^*(G|H)}= \min_{w_v\geq 0, \sum_{v}w_v=1} \max_{f}\sum_{v\in V(G)}w_v|f(v)|.$$

It remains to prove the first and fourth parts of the theorem. Below, we provide a proof for these. An alternative proof is given in Appendix B. We begin with the following lemma.
\begin{lemma}\label{lemmaA1}
To compute $\alpha^*(G|H)$, it suffices to take supremum over graphs $W$ with the following structure: the vertex set of $W$ partitions as follows: $\mathcal{V}({W})=\bigcup_{\mathcal S\in \mathcal I(G)} \mathcal{W}_{\mathcal S}$ for some disjoint sets $\mathcal{W}_{\mathcal S}$. The edge structure in $W$ is as follows: for any $\mathcal S_1, \mathcal S_2\in \mathcal I(G)$ vertex $u_1\in \mathcal{W}_{S_1}$ is connected to $u_2\in \mathcal{W}_{\mathcal S_2}$ iff $\mathcal S_1$ and $\mathcal S_2$ are disconnected in $G$. In particular, if $\mathcal S_1=\mathcal S_2$ and $u_1, u_2\in \mathcal{W}_{\mathcal S}$, then $u_1$ and $u_2$ are not connected. Moreover, we have
$$\alpha(W \boxtimes G)=\sum_{\mathcal S\in \mathcal I(G)}|\mathcal{W}_{\mathcal S}|\times |\mathcal S|.$$
\end{lemma}
\begin{proof}

Take some arbitrary graph $W$. Let $\mathcal B\subseteq \mathcal{V}(W)\times  \mathcal{V}(G)$ be an independent set for $W \boxtimes G$ of maximum size. For every $u\in \mathcal{V}(W)$, let
$\mathcal B_{u}=\{v\in\mathcal{V}(G): (u,v)\in \mathcal B\}.$
If $\mathcal B_{u}=\emptyset$ for some vertex $u$ of $W$, we can simply remove this vertex from $W$. This removal would not affect $\alpha(W\boxtimes G)$ but can potentially reduce $\alpha(W\boxtimes H)$. Therefore, without loss of generality, we can assume that $\mathcal B_{u}\neq \emptyset$ for all vertices $u$ of $W$.

Since $\mathcal B$ is an independent set for $W \boxtimes G$, for every $u\in \mathcal{V}(W)$, $\mathcal B_u$ must be an independent set for $G$. Thus, $\mathcal B_u\in \mathcal I(G)$. 
This leads to the following partition of the vertex set $\mathcal{V}(W)$: for every non-empty $\mathcal S\in \mathcal I(G)$, let
$\mathcal{W}_{\mathcal S}=\{u\in \mathcal{V}(W): \mathcal B_u=\mathcal S\}.$
With this definition, the condition that $\mathcal B\subseteq \mathcal{V}(W)\times  \mathcal{V}(H)$ is an independent set for $W \boxtimes G$ is equivalent with the following condition: take $u_1\in \mathcal{W}_{\mathcal S_1}$ and $u_2\in \mathcal{W}_{\mathcal S_2}$. If $\mathcal S_1$ and $\mathcal S_2$ are not disconnected (\emph{i.e.,} either $\mathcal S_1\cap \mathcal S_2\neq \emptyset$ or there is an edge in $G$ between a vertex in $\mathcal S_1$ and a vertex in $\mathcal S_2$) then there should not be any edge between $u_1, u_2$ in $W$. If
$\mathcal S_1$ and $\mathcal S_2$ are disconnected; there may or may not be edges between vertices in $\mathcal S_1$ and $\mathcal S_2$. We claim that, without loss of generality, we can add edges between the vertices in $\mathcal S_1$ and $\mathcal S_2$ if $\mathcal S_1$ and $\mathcal S_2$ are disconnected since their addition would not affect $\alpha(W \boxtimes G)$ but could potentially decrease $\alpha(W \boxtimes H)$, which is desirable as we want to maximize the ratio $\alpha(W \boxtimes G)/\alpha(W \boxtimes H)$. Thus, the graph $W$ will have the form given in the lemma statement. We also have
$$|\mathcal B|=\sum_{u\in \mathcal{V}(W)}|\mathcal B_u|=\sum_{\mathcal S\in \mathcal I(G)}|\mathcal{W}_{\mathcal S}|\times | \mathcal S|.$$
\end{proof}

Consider a graph $W$ with the structure in Lemma \ref{lemmaA1}. We claim that 
\begin{align}\alpha(W\boxtimes H)=\max_{f\in\mathcal{F}} \sum_{\mathcal S\in \mathcal I(G)} |\mathcal{W}_{\mathcal S}| \times |f(\mathcal S)|\label{eqnA1d}\end{align}
where $\mathcal{F}$ is the set of all functions $f: \mathcal I(G)\mapsto \mathcal I(H)$ such that for any distinct $\mathcal S_1,\mathcal \mathcal S_2\in \mathcal I(G)$, if $\mathcal S_1$ and $\mathcal S_2$ are disconnected in $G$ then $f(\mathcal S_1)$ and $f( \mathcal S_2)$ are disconnected in $H$. 
 First, observe that 
$\alpha(W\boxtimes H)\geq \max_{f\in\mathcal{F}} \sum_{\mathcal S\in \mathcal I(G)} |\mathcal{W}_{\mathcal S}| \times |f(\mathcal S)|.$
To see this, for every function $f\in\mathcal{F}$, 
$\{(u, v)\big| u\in \mathcal{W}_{\mathcal S},  v\in f(\mathcal S)\text{ for some } \mathcal S\in \mathcal I(G)  \}$
is an independent set. On the other hand, take an independent set $\mathcal A$ for $W\boxtimes H$ of maximum size. Take some $\mathcal S\in \mathcal I(G)$ and for every $u\in\mathcal{\mathcal W}_{\mathcal S}$ consider 
$\mathcal A_{u}=\{v\in\mathcal{V}(H): (u,v)\in \mathcal A\}.$
Since $\mathcal A$ is an independent set for $W\boxtimes H$, we have $\mathcal A_u\in \mathcal I(H)$.
Consider $|\mathcal A_{u}|$ for $u\in\mathcal{W}_{\mathcal S}$, and let $u^*\in\mathcal{W}_{\mathcal S}$ have maximum $|\mathcal A_{u^*}|$. For all $u\in\mathcal{W}_{\mathcal S}$, let us replace $\mathcal A_u$ by $\mathcal A_{u^*}$. This transformation will not decrease the size of $\mathcal A$ (as $|\mathcal A_{u^*}|$ had maximum size) and will keep $\mathcal A$ as an independent set for $W\boxtimes H$. If we make this transformation on all  $\mathcal S\in \mathcal I(G)$, we can force all $u\in\mathcal{W}_{\mathcal S}$ mapped to the same independent set in $\mathcal I(H)$, which we can call $f(\mathcal S)\in \mathcal I(H)$. Thus,  the size of the independent set $\mathcal A$ is no more than $\sum_{\mathcal S\in \mathcal I(G)} |\mathcal{W}_{\mathcal S}| \times |f(\mathcal S)|$. Hence, 
$\alpha(W\boxtimes H)\leq \max_{f\in\mathcal{F}} \sum_{\mathcal S\in \mathcal I(G)} |\mathcal{W}_{\mathcal S}| \times |f(\mathcal S)|.$ This implies \eqref{eqnA1d}. To sum this up, we obtain
$$\frac{\alpha(W \boxtimes G)}{\alpha(W\boxtimes H)}=\frac{\sum_{\mathcal S\in \mathcal I(G)}|\mathcal{W}_{\mathcal S}|\times |\mathcal S|}{\max_{f\in\mathcal{F}} \sum_{\mathcal S\in \mathcal I(G)} |\mathcal{W}_{\mathcal S}| \times |f(\mathcal S)|}.$$
Observe that $|\mathcal{W}_{\mathcal S}|$ can be set to any arbitrary natural number. Letting $\omega_{\mathcal S}=|\mathcal{W}_{\mathcal S}|$ we  obtain 
$$\sup_{W}\frac{\alpha(W \boxtimes G)}{\alpha(W\boxtimes H)}=\sup_{\{\omega_{\mathcal S}\in \mathbb{Z}\cup\{0\}: \mathcal S\in \mathcal I(G)\}}~\min_{f\in\mathcal{F}}\frac{\sum_{\mathcal S\in \mathcal I(G)}\omega_{\mathcal S}|\mathcal S|}{\sum_{\mathcal S\in \mathcal I(G)} \omega_{\mathcal S}|f(S)|}.$$
In particular, if the supremum on the right-hand side is a maximum, the supremum on the left-hand side is also a maximum. Next, we claim that to compute
\begin{align}\sup_{\{\omega_{\mathcal S}\in \mathbb{Z}\cup\{0\}: \mathcal S\in \mathcal I(H)\}}~\min_{f\in\mathcal{F}}\frac{\sum_{\mathcal S\in \mathcal I(H)}\omega_{\mathcal S}|\mathcal S|}{\sum_{\mathcal S\in \mathcal I(H)} \omega_{\mathcal S}|f(\mathcal S)|},\label{eqnABC}
\end{align}
without loss of generality, one can assume that $\omega_{\mathcal S}=0$ if $|\mathcal S|\neq 1$. We first show that  without loss of generality, to compute \eqref{eqnABC},  we can restrict the minimization over $f\in\mathcal{F}$ satisfying
$$|f(\mathcal S)|\geq \sum_{i\in \mathcal S}|f(\{i\})|,$$
for any set $\mathcal S\in \mathcal I(G)$.
 To show this, take an arbitrary function $f$ and  $\mathcal S\in \mathcal I(G)$ such that
$$|f(\mathcal S)|< \sum_{i\in \mathcal S}|f(\{i\})|.$$
Observe that $f(\{i\})$ and $f(\{j\})$ for $i,j\in \mathcal S$ are disconnected independent sets in $\mathcal I(H)$, since $\{i\}$ and $\{j\}$ are disconnected independent sets in $\mathcal I(G)$. Thus,
$$\sum_{i\in \mathcal S}|f(\{i\})|=\big|\cup_{i\in \mathcal S}f(\{i\})\big|.$$
We claim that replacing $f(\mathcal S)$ by $\cup_{i\in \mathcal S}f(\{i\})$ would not violate conditions on $f$. This change would increase $|f(\mathcal S)|$ and hence would also increase the term $\sum_{\mathcal S\in \mathcal I(G)} \omega_{\mathcal S}\cdot |f(\mathcal S)|$. To show that the replacement would not violate conditions on $f$, take some arbitrary set $\mathcal T\in \mathcal I(G)$ where $\mathcal T$ and $\mathcal S$ are disconnected in $G$. Then, the set $\{i\}$ and $\mathcal T$ are disconnected in $G$, implying that $f(\mathcal T)\in \mathcal I(H)$ is disconnected from $f(\{i\})$. Hence, $f(\mathcal T)$ is disconnected from $\cup_{i\in \mathcal S}f(\{i\})$. Thus, the desired condition is satisfied by this replacement. To sum this up, without loss of generality, we can restrict the minimization over $f$ to those satisfying
$$|f(\mathcal S)|\geq \sum_{i\in \mathcal S}|f(\{i\})|.$$
Now, take some arbitrary set of $\{\omega_{\mathcal S}\geq 0: \mathcal S\in \mathcal I(G)\}$. Then, consider the following:
\begin{align}\omega'(\{i\})&=\sum_{\mathcal S: i\in \mathcal S}\omega_{\mathcal S}.\\
\omega'(\mathcal S)&=0\qquad \text{ if }|\mathcal S|>1.\end{align}
Then, we have
\begin{align}\sum_{\mathcal S\in \mathcal I(G)}\omega'_{\mathcal S}\cdot |S|&=
\sum_{i}\sum_{\mathcal S:~i\in \mathcal S} \omega_{\mathcal S}
\\&=
\sum_{\mathcal S\in \mathcal I(G)}\omega_{\mathcal S}\cdot |\mathcal S|
\end{align}
and
\begin{align}
\sum_{\mathcal S\in \mathcal I(G)} \omega'_{\mathcal S}\cdot |f(\mathcal S)|&=\sum_{i}\sum_{\mathcal S:~i\in \mathcal S} \omega_{\mathcal S}\cdot |f(\{i\})|
\\&
=\sum_{\mathcal S\in \mathcal I(G)}\omega_{\mathcal S}\sum_{i\in \mathcal \mathcal S}  |f(\{i\})|
\\&
\leq \sum_{\mathcal S\in \mathcal I(G)} \omega_S\cdot |f(\mathcal S)|.
\end{align}
Since this holds for every arbitrary function $f$,  this implies that without loss of generality, we can assume that $\omega_{\mathcal S}$ is non-zero only when $|\mathcal S|=1$. Assuming that $  \mathcal V (G)=\{v_1,v_2,\cdots, v_k\}$, let $w_i=\omega_{\{i\}}$, and $\mathcal T_i=f(\{i\})$. Then, 
$$\sup_{W}\frac{\alpha(W \boxtimes G)}{\alpha(W\boxtimes H)}=\sup_{\{w_i\in \mathbb{Z}\cup\{0\}\}}~\min_{\mathcal T_1, \cdots, \mathcal T_k}\frac{\sum_{i=1}^kw_i}{\sum_{i=1}^kw_i |\mathcal T_i|}$$
where the minimum is over any collection of sets $\mathcal T_1, \cdots, \mathcal T_k\in \mathcal I(H)$ such that $\mathcal T_i$ and $\mathcal T_j$ are disconnected in $H$ if there is no edge between $v_i$ and $v_j$ in $G$. 

Observe that scaling $w_i$ would not change the above expression. Therefore, 
$$\sup_{W}\frac{\alpha(W \boxtimes G)}{\alpha(W\boxtimes H)}=\sup_{\{w_i\in \mathbb{Q}, w_i\geq 0\}}~\sum_{i=1}^kw_i$$
subject to $\sum_{i=1}^kw_i |\mathcal T_i|\leq 1$
for any collection of sets $\mathcal T_1, \cdots,\mathcal T_k\in \mathcal I(H)$ where $\mathcal T_i$ and $\mathcal T_j$ are disconnected if $v_iv_j\notin \mathcal E(G)$. Relaxing $w_i\in \mathbb{Q}$ to $w_i\in \mathbb{R}$ would not change the value of this linear program because the solution of this linear program is rational. The first part of the theorem thus follows. To show the last part, observe that the domain of the linear program is compact since $w_i\leq 1$  for all $i$. To see this, take $\mathcal T_i$ as an independent set in $H$ and $\mathcal T_j=\emptyset$ for $j\neq i$. The compactness of the domain of the linear program implies that the solution is obtained at some finite $(w_1, \cdots, w_k)$. Thus, the supremum in the definition of $\alpha^*(G|H)$ is a maximum.

\subsection{Proof of Theorem \ref{thm2}} 
 \label{proofthm2}

For the first part, let $C$ be the maximizer for $\alpha^*(G|W)$. Then, we have
    \begin{equation*}
        \frac{\alpha^*(G|W)}{\alpha^*(H|W)} = \frac{\sup_{A} \frac{\alpha(G \boxtimes A)}{\alpha(W \boxtimes A)}}{\sup_{B} \frac{\alpha(H \boxtimes B)}{\alpha(W \boxtimes B)}}
        = \frac{\frac{\alpha(G \boxtimes C)}{\alpha(W \boxtimes C)}}{\sup_{B} \frac{\alpha(H \boxtimes B)}{\alpha(W \boxtimes B)}}
        \leq \frac{\frac{\alpha(G \boxtimes C)}{\alpha(W \boxtimes C)}}{\frac{\alpha(H \boxtimes C)}{\alpha(W \boxtimes C)}} = \frac{\alpha(G \boxtimes C)}{\alpha(H \boxtimes C)} \leq \alpha^*(G|H).
    \end{equation*}
    Similarly, let $D$ be the maximizer for $\alpha^*(H|W)$. We have
    \begin{equation*}
        \frac{\alpha^*(G|W)}{\alpha^*(H|W)} = \frac{\sup_{A} \frac{\alpha(G \boxtimes A)}{\alpha(W \boxtimes A)}}{\sup_{B} \frac{\alpha(H \boxtimes B)}{\alpha(W \boxtimes B)}}= \frac{\sup_{A} \frac{\alpha(G \boxtimes A)}{\alpha(W \boxtimes A)}}{ \frac{\alpha(H \boxtimes D)}{\alpha(W \boxtimes D)}} \geq \frac{\frac{\alpha(G \boxtimes D)}{\alpha(W \boxtimes D)}}{\frac{\alpha(H \boxtimes D)}{\alpha(W \boxtimes D)}} = \frac{\alpha(G \boxtimes D)}{\alpha(H \boxtimes D)} \geq \frac{1}{\alpha^*(H|G)}. 
    \end{equation*}

For the second part, consider the max form of the linear programs (first part of Theorem \ref{thm1}) for computing $\alpha^*(G_1+G_2|H)$, $ \alpha^*(G_1|H)$ and $ \alpha^*(G_2|H)$  and let us denote the set of constraints in these linear programs by $Cons(G_1+G_2|H),Cons(G_1|H)$ and $Cons(G_2|H)$, respectively. Thus, for instance, $\alpha^*(G_1+G_2|H)=\max \sum_{v_i\in \mathcal V(G_1)} w_i+\sum_{v_i\in \mathcal V(G_2)} w_i$ subject to $Cons(G_1+G_2)$. Also $ \alpha^*(G_1|H)=\max \sum_{v_i\in \mathcal V(G_1)} w_i$ subject to  $Cons(G_1|H)$ and $\sum_{v_i\in \mathcal V(G_2)} w_i$ subject to $Cons(G_2|H)$. 
It is clear that, $Cons(G_1|H)\subset Cons(G_1+G_2|H)$ and $Cons(G_2|H)\subset Cons(G_1+G_2|H)$, so
$\alpha^*(G_1+G_2|H)\leq \alpha^*(G_1|H)+\alpha^*(G_2|H).$

For the third part, it is clear that $\alpha^*((G_1+G_2)^c|H)\geq\max(\alpha^*(G_1^c|H),\alpha^*(G_2^c|H))$.
Consider the characterization in the third part of Theorem \ref{thm1} to show the reverse direction. Take the optimal random mappings for $\alpha^*(G_1^c|H)$ and $\alpha^*(G_2^c|H)$ and combine them to construct a random mapping for $(G_1+G_2)^c$. This random mapping implies that
$$\alpha^*((G_1+G_2)^c|H)\leq \max(\alpha^*(G_1^c|H),\alpha^*(G_2^c|H)).$$

For the fourth part, we have
\begin{align*}
        \frac{1}{\alpha^*(G|H_1+H_2)} &= \inf_{W} \frac{\alpha((H_1 + H_2) \boxtimes W)}{\alpha(G \boxtimes W)} = \inf_{W} \frac{\alpha(H_1 \boxtimes W) + \alpha(H_2 \boxtimes W)}{\alpha(G \boxtimes W)} 
        \\& \geq \inf_{W} \frac{\alpha(H_1 \boxtimes W)}{\alpha(G \boxtimes W)} + \inf_{W} \frac{\alpha(H_2 \boxtimes W)}{\alpha(G \boxtimes W)} = \frac{1}{\alpha^*(G|H_1)} +  \frac{1}{\alpha^*(G|H_2)}.
\end{align*}
According to Theorem \ref{alphavertran}, when $G$ is a vertex-transitive graph on $k$ vertices, we have the following:
\begin{align*}
        \inf_{W} \frac{\alpha((H_1+H_2) \boxtimes W)}{\alpha(G \boxtimes W)} &= \frac{\alpha((H_1+H_2) \boxtimes G^c)}{k} = \frac{\alpha(H_1 \boxtimes G^c)}{k} + \frac{\alpha(H_2 \boxtimes G^c)}{k} 
        \\&= \inf_{W} \frac{\alpha(H_1 \boxtimes W)}{\alpha(G \boxtimes W)} + \inf_{W} \frac{\alpha(H_2 \boxtimes W)}{\alpha(G \boxtimes W)}= \frac{1}{\alpha^*(G|H_1)} +  \frac{1}{\alpha^*(G|H_2)}. 
\end{align*}
Thus, equality holds in this case.

\subsection{Proof of Lemma \ref{w}}
 \label{prooflemm1}
Take some $W$ such that $\alpha^*(G_1\boxtimes G_2\boxtimes \cdots \boxtimes G_r)=\frac{\alpha(G_1\boxtimes G_2\boxtimes \cdots \boxtimes G_r\boxtimes W)}{\alpha(W)}$.

We show that this choice of $W$ works for us. Note that
\begin{align}\nonumber
\prod_{i=1}^r\alpha^*(G_i)&=
    \alpha^*(G_1\boxtimes G_2\boxtimes \cdots \boxtimes G_r)\nonumber=\frac{\alpha(G_1\boxtimes G_2\boxtimes \cdots \boxtimes G_r\boxtimes W)}{\alpha(W)}
    \\&=\nonumber
    \frac{\alpha( G_2\boxtimes \cdots \boxtimes G_r\boxtimes G_1\boxtimes W)}{\alpha(G_1\boxtimes W)}\cdot
    \frac{\alpha(G_1\boxtimes W)}{\alpha(W)}
    \\&\leq \label{eqnonestepineq}
        \max_{W'}\frac{\alpha(G_2\boxtimes \cdots \boxtimes G_r\boxtimes  W')}{\alpha(W')}\cdot
    \max_{W'}\frac{\alpha(G_1\boxtimes W')}{\alpha(W')}
    \\&=\alpha^*(G_2\boxtimes \cdots \boxtimes G_r)\alpha^*(G_1)\nonumber=\prod_{i=1}^r\alpha^*(G_i).\nonumber
\end{align}
Therefore, equality must hold in \eqref{eqnonestepineq}. Hence 
$\alpha^*(G_1)=\frac{\alpha(G_1\boxtimes W)}{\alpha(W)}.$
A similar argument establishes the desired equality for $G_i$, $i>1$.

\subsection{Proof of Theorem \ref{thm3}}
 \label{proofthm3}

It suffices to show that
\begin{align}
    \alpha^*(G|H) \geq \frac{X(G)}{X(H)}.\label{eqnjjj}
\end{align}
The other inequality $\frac{1}{\alpha^*(H|G)} \leq \frac{X(G)}{X(H)}$ is equivalent with $\alpha^*(H|G) \geq \frac{X(H)}{X(G)}$, which is the same inequality as in \eqref{eqnjjj} with $G$ and $H$ swapped.

Observe that by choosing $W$ as a trivial graph with just one node and no edges, we have
$$\alpha^*(G|H)=\sup_{W}\frac{\alpha(G\boxtimes W)}{\alpha(H\boxtimes W)}\geq \frac{\alpha(G)}{\alpha(H)}.$$

Next, assume that we have two quantities $X(\cdot)$ and $Y(\cdot)$ on a graph such that \begin{align}\alpha(G\boxtimes W)&\leq X(G)Y(W)\label{eqnA12}\end{align}
for any two graphs $G$ and $W$. Then, we have
\begin{align*}\alpha^*(G|H)&=\sup_{W}\frac{\alpha(G\boxtimes W)}{\alpha(H\boxtimes W)}
\geq 
\sup_{W}\frac{\alpha(G\boxtimes W)}{X(H)Y(W)}
= 
\frac{1}{X(H)}\sup_{W}\frac{\alpha(G\boxtimes W)}{Y(W)}.
\end{align*}
Consider the choice of $X(G)=Y(G)=\vartheta(G)$, the Lov\'{a}sz number of $G$. Then  \eqref{eqnA12} is satisfied. Moreover, as shown in \cite[Theorem 2]{Winter},
$$\sup_{W}\frac{\alpha(G\boxtimes W)}{\vartheta(W)}=\vartheta(G).$$
Thus, the desired inequality holds for the Lov\'{a}sz number of $G$.

Similarly, the choices of $X(G)=\vartheta^{-}(G), Y(G)=\vartheta^{+}(G)$ and $X(G)=\vartheta^{+}(G), Y(G)=\vartheta^{-}(G)$. Then, by \cite[Lemma 7]{Winter}, \eqref{eqnA12} is satisfied. From \cite[Theorem 8]{Winter}, we have
$$\sup_{W}\frac{\alpha(G\boxtimes W)}{\vartheta^{-}(W)}=\vartheta^{+}(G),$$
$$\sup_{W}\frac{\alpha(G\boxtimes W)}{\vartheta^{+}(W)}=\vartheta^{-}(G).$$
So, we obtain the desired results for the Schrijver's or Szegedy's variants of the Lov\'{a}sz number.

\subsection{Proof of Theorem \ref{thmN4new}}
\label{proofthmN4new}

Note that 
$\mathcal{S}$ is an independent set of $(G')^c\boxtimes H'$. 
    For any vertex $u$ in $G'$, define $$\mathcal A_u:=\{v\in \mathcal V(H'): (u,v)\in \mathcal S\}$$ We similarly define $\mathcal B_v$ for vertices $v$ in $H'$
as
$$\mathcal B_v:=\{u\in \mathcal V(G'): (u,v)\in \mathcal S\}.$$
By assumption about $G'$ and $H'$, the sets $A_u$ and $B_v$ are non-empty for any $u \in G'$ and $v\in H'$. For any vertex $v$ of $H'$, $\mathcal B_v$ is an independent set in $(G')^c$. Therefore, it forms a clique in $G'$. To construct $G'$ from $H'$ by applying the expansion operations, we first replace every vertex $v$ in $H'$ with a clique of size $|\mathcal B_v|$ to obtain $H'_1$.
    We can imagine the vertices of these cliques labeled by the members of $\mathcal B_v$. 
    
    If two vertices in $H'_1$ labeled by $u_i$ and $u_j$ are connected by an edge in $H'_1$, then either for a vertex $v$ in $H'$ we have $u_i, u_j \in \mathcal B_v$, or there exist two connected vertices $v_i, v_j \in H'$ such that $u_i \in \mathcal B_{v_i}$ and $u_j \in \mathcal B_{v_j}$. Since $\mathcal B_v$ for $v\in \mathcal{V}(H')$ is obtained from an independent set of $(G')^{c} \boxtimes H'$, it follows that in both cases, $u_i$ must be disconnected from $u_j$ in $(G')^{c}$. Consequently, there is an edge between them in $G'$. To sum this up, if there is an edge between two vertices labeled by $u_i$ and $u_j$ in $H'_1$, then there is an edge between $u_i$ and $u_j$ in $G'$.
    
    Finally, using Remark \ref{mrk:merge}, we can merge all the vertices in $H'_1$ that have the same label to obtain $H''$. The assumption that $\mathcal A_u$ is non-empty would imply that the vertices of $H''$ cover all vertices of $G'$. This yields a graph $H''$ which is a spanning subgraph of $G'$. Now, by adding the missing edges of $G'$ to $H''$, we can finish the construction of $G'$.

\subsection{Proof of Theorem \ref{alphavertran}}
\label{proofthmalphavertran}

We establish the following statement:
for any graphs $G$ and $H$, where $\mathcal V(G)=\{v_1,\dots, v_k\}$, we have
\begin{align}\alpha^*(G|H)\geq k\min \frac{1}{\sum_{i=1}^k  |\mathcal T_i|}=\frac{|\mathcal V(G)|}{\alpha(G^c\boxtimes H)},\label{aseqrq2}\end{align}
where the minimum is over all collection of sets $\mathcal T_1, \cdots, \mathcal T_k\in \mathcal I(H)$ such that $\mathcal T_i$ and $\mathcal T_j$ are disconnected in $H$ if there is no edge between $v_i$ and $v_j$ in $G$. Moreover, the lower bound in \eqref{aseqrq2} is tight (holds with equality) if $G$ is vertex-transitive.

According to Theorem \ref{thm1}, 
$\alpha^*(G|H)=\max_{\mathbf w} \sum_{i=1}^k w_i$,
where the maximum is over all non-negative weights $w_i$ such that $\sum w_i|\mathcal T_i|\le 1$ for all collection $\mathcal T_1,\dots, \mathcal T_k\in \mathcal I(H)$ such that $\mathcal T_i$ and $\mathcal T_j$ are disconnected if $v_iv_j\not\in \mathcal E(G)$. If we only restrict to those weights that $w_1=\dots=w_k$ we get
$$\alpha^*(G|H)\ge k \max w_1$$
where $w_1\le  1/(\sum_i |\mathcal T_i|)$. Since $\max w_1=\min 1/(\sum_i |\mathcal T_i|)$ the lower bound is established. Next, we show that $\max\sum_{i=1}^k  |\mathcal T_i|=\alpha(G^c\boxtimes H)$. Let $\mathcal T^*_i$ be a valid assignment of independent sets of $H$ to the vertices of $G$ such that $\sum_{i=1}^k  |\mathcal T^*_i| =\max \sum_{i=1}^k  |\mathcal T_i|$. Consider all the pairs of vertices $v_i,u_j$, such that $v_i\in \mathcal V(G)$ and $u_j\in \mathcal T^*_i$. The set of these pairs are independent in $G^c\boxtimes H$, so, $\max\sum_{i=1}^k  |\mathcal T_i|\leq \alpha(G^c\boxtimes H)$. On the other hand, if $(v^*_i,u^*_j)$'s are the set of vertices in a maximum independent set, $\mathcal I^*$, of $G^c\boxtimes H$, then if we define $\mathcal T_i\triangleq\{u^*_j|(v^*_i,u^*_j)\in \mathcal I^*\}$, we get a valid assignment. This shows  $\max\sum_{i=1}^k  |\mathcal T_i|\ge \alpha(G^c\boxtimes H)$.
Thus, $\max\sum_{i=1}^k  |\mathcal T_i|=\alpha(G^c\boxtimes H)$.

Next, assume that $G$ is vertex-transitive. For any pair
 of vertices $v_i$ and $v_j$, let $\Pi_{i,j}$ be the set of automorphisms  $\phi:G\rightarrow G$ that maps vertex $v_i$ to vertex $v_j$. Since $G$ is vertex-transitive, $|\Pi_{i,j}|\geq 1$. We claim that $|\Pi_{i,j}|=|\Pi_{1,1}|$ does not depend on $i,j$. To see this, take some $\phi^*_1\in \Pi_{1,i}$ and $\phi^*_2\in\Pi_{j,1}$. Then, $\phi\in \Pi_{i,j}$ if and only if $\phi^*_2\circ \phi\circ \phi^*_1\in \Pi_{1,1}$. This provides a one-to-one map between $\Pi_{i,j}$ and $\Pi_{1,1}$ and the claim follows.

Let $\pi$ be a permutation on $\{1,\dots, k\}$ such that
$\phi(v_i)=v_{\pi(i)}$ is
an automorphism of $G$. If $\mathcal T_1,\dots, \mathcal T_k$ is a collection of independent subsets of $\mathcal I(H)$ such that $\mathcal T_i$ and $\mathcal T_j$ are disconnected if $v_iv_j\not\in \mathcal E(G)$, then the constraint $\sum w_{\pi(i)}|\mathcal T_i|\le 1$ must hold as well. Now if we take the average of all the inequalities $\sum w_{\pi(i)}|\mathcal T_i|\le 1$ over all automorphisms (over all $\pi$), we get
$\bar{w}(\sum_i |\mathcal T_i|)\le 1$,
where $\bar{w}=(w_1+\dots+w_k)/k$.  So 
$$\alpha^*(G|H)=\max_w (w_1+\dots+w_k)=k\max \bar{w}\le k \frac{1}{\sum_i |\mathcal T_i|}$$
for any valid collection $\mathcal T_1,\dots,\mathcal T_k\in\mathcal I(H)$, so
$\alpha^*(G|H)\le k \min \frac{1}{\sum_i |\mathcal T_i|}$
and the lemma holds.

\subsection{Proof of Theorem \ref{Thm:bellowcayley}}
\label{proofthmCayley}
Assume that $\alpha^{*}(G|H) \leq 1$. We aim to prove that $G \in Expand(H)$. 

\subsubsection{Proof of the first part}

Since $G$ is a Cayley graph, it is vertex-transitive. Therefore, by Theorem \ref{alphavertran}, we have:
\[
\alpha^{*}(G|H) = \frac{|\mathcal{V}(G)|}{\alpha(G^{c} \boxtimes H)}
\]
Thus, $\alpha^{*}(G|H) \leq 1$ implies that $|\mathcal{V}(G)| \leq \alpha(G^{c} \boxtimes H)$. Let $\mathcal S$ be a maximum independent set $G^{c} \boxtimes H$. For any vertex $u$ in $G$, define $$\mathcal A_u:=\{v\in \mathcal V(H): (u,v)\in \mathcal S\}$$ We similarly define $\mathcal B_v$ for vertices $v$ in $H$
as
$$\mathcal B_v:=\{u\in \mathcal V(G): (u,v)\in \mathcal S\}.$$
 %We call this sets the projections of $\mathcal{S}$ to $v$.
We have
\[
\alpha(G^{c} \boxtimes H) = \sum\limits_{u \in G} |\mathcal A_u| \geq |\mathcal{V}(G)|.
\]
If $|A_u| \geq 1$ for all $u \in G$, we can apply Theorem \ref{thmN4new} to deduce that we can construct $G$ from $H$ by applying the expansion operations. Thus, assume that $|A_u| =0$ for some $u \in G$. Label vertices of $G=Cay(\mathbb{Z}_n, \pm 1,\pm 2,\dots \pm k)$ by $u_1, u_2, \cdots, u_n$ where either $n=|\mathcal{V}(G)|=c(k+1)+1$ or $n=|\mathcal{V}(G)|=c(k+1)$ for some integer $c$.    
    Without loss of generality assume $|A_{u_n}| =0$. Now, consider the following $k+1$ sets in $G$:
    \begin{equation*}
        \mathcal S_i = \lbrace u_{t} \, | \, t \equiv i \pmod {k+1}), \, \, t \neq n \rbrace, \qquad 1\leq i\leq k+1.
    \end{equation*}
    We have $\bigcup_{i=1}^{k+1}\mathcal{S}_i=\mathcal{V}(G)-u_n$. Each set $\mathcal S_i$ is an independent set in $G$.

    Since $\sum\limits_{u \in G} |\mathcal A_u| \geq n$, we must have $\sum\limits_{u \in \mathcal S_{i^*}} |\mathcal A_u| \geq \ \lceil \frac{n}{k+1} \rceil$ for some integer $1 \leq i^* \leq k+1$. We construct $G$ from $H$ by applying the expansion operations as follows: we first delete every vertex $v$ from $H$ such that  $\mathcal B_v \cap S_{i^*} = \emptyset$  to obtain $H'$. Since $S_{i^*}$ was an independent set in $G$, it forms a clique in $G^c$. Since the projections $\mathcal B_v$ are coming from an independent set of $G^c \boxtimes H$, all the remaining vertices in $H'$ must be disconnected from each other, and $H'$ should not have any edges.
    
     Next, due to our assumption about the set $S_{i^*}$, $H'$ has at least $\lceil \frac{n}{k+1} \rceil$ vertices. To obtain $H''$, we replace $\lfloor \frac{n}{k+1} \rfloor$ of these vertices with cliques of size $k+1$. Thus, if $n=c(k+1)$, the graph $H''$ will consist of $c=\lfloor \frac{n}{k+1} \rfloor$ disjoint cliques of size $k+1$. If $n=c(k+1) + 1$, the graph $H''$ will consist of $c$ disjoint cliques of size $k+1$, plus one single vertex. Hence, $H''$ is a spanning subgraph of $G$ in either case. By adding the missing edges of $H''$ to it, we complete the construction of $G$. Thus, the theorem is proved for this case.

\subsubsection{Proof of the second part}
    Next, we show that the conjecture holds for perfect graphs. Assume that $G$ is a perfect graph with $n$ vertices and $\alpha^{*}(G|H) \leq 1$. We aim to prove that $G \in Expand(H)$.
    Since from Theorem \ref{alphavertran} we have
    $$\alpha^{*}(G|H)\geq \frac{|\mathcal{V}(G)|}{\alpha (H \boxtimes G^c)}$$
    we obtain $|\mathcal{V}(G)| \leq \alpha (H \boxtimes G^c)$.  
        By the Perfect Graph Theorem, we know that $G^c$ is also perfect. Therefore, we have $|\mathcal{V}(G)|  \leq \alpha (H \boxtimes G^c)=\alpha(H)\alpha(G^c)$. Let $\mathcal{S}$ be a maximum independent set of $H \boxtimes G^c$ such that $\mathcal{S}=\mathcal{S}_1\times \mathcal{S}_2$ where $\mathcal{S}_1$ and $\mathcal{S}_2$ are maximum independent sets of $H$ and $G^c$ respectively. 

        First, we delete every vertex $u \notin \mathcal{S}_1$ from $H$. Secondly, replace the remaining vertices of $H$ with a clique of size $\alpha(G^c)$ to obtain $H'$. Due to the selection of a maximum independent set, $H'$ consists of $\alpha(H)$ disjoint cliques, each of size $\alpha(G^c)$.         Since $\alpha^{*}(G|H) \leq 1$, we should have ${\alpha(G)}/{\alpha(H)} \leq 1$, therefore $\alpha(G) \leq \alpha(H)$ and $H'$ has at least that $\alpha(G)$ disjoint cliques.
        
        In Lemma \ref{lemvperf} (provided below), we show that the vertices of a perfect graph $G$ can be partitioned into $\alpha(G)$ disjoint sets such that the induced subgraph on each set is a clique. Since cliques of $G$ are independent sets of $G^c$, the cliques have a maximum size of $\alpha(G^c)$. Noting that $H'$ consists of at least $\alpha(G)$ disjoint cliques, we can remove some vertices from $H'$ to obtain $H''$ such that $H''$ precisely represents the clique partition of $G$. This implies that $H''$ is a spanning subgraph of $G$, and by adding the remaining edges to $H''$, we can complete the construction of $G$ from $H$.

       \begin{lemma}\label{lemvperf}
            The vertices of a perfect graph $G$ can be partitioned into $\alpha(G)$ disjoint sets such that the induced subgraph on each set is a clique. 
        \end{lemma}
        \begin{proof}
            Since $G$ is perfect, by the Perfect Graph Theorem, $G^c$ is also perfect. This means the maximum clique size of $G^c$, which is the same as $\alpha(G)$, is equal to its chromatic number. A coloring of $G^c$ is equivalent to partitioning the vertices of $G^c$ into independent sets, and equivalently vertices of $G$ into cliques. Because $\chi(G^c) = \alpha(G)$, this partition consists of $\alpha(G)$ cliques. Hence, the lemma is proven.
        \end{proof}

\subsection{
Proof of Theorem \ref{cayex}
}\label{proofcayex}

Note that for a vertex-transitive graph $G$ of size $n$, $\alpha^*(G)=\frac{n}{\omega(G)}$, where $\omega(G)$ is the clique number of $G$. Thus, $\alpha^*(G)=\frac{n}{k+1}$ and $\alpha^*(H)=\frac{m}{k+1}$. It follows that $\alpha^*(G|H)\ge \frac{n}{m}$ and hence $\alpha(G^c\boxtimes H)\le m$. 
If there is a homomorphism from $H$ to $G$ then by Lemma 5, we have $\alpha(G^c\boxtimes H)\ge m$ and hence we find that $\alpha^*(G|H)=\frac{n}{m}$. We show that such a homomorphism exists if there are integers $\ell,s\ge 0$ such that
$$m=\ell n+s(k+1).$$
Let $\{1,\dots,n\}$ and $\{1,\dots, m\}$ be the vertex sets for $G$ and $H$. If $\ell=0$ define $\mathsf{g}(i(k+1)+j)=j$ where $1\leq j\le k+1$ and $i\ge 0$. If $\ell>0$, define $\mathsf{g}(in+j)=j$ for $1\le j\le n$ and $i=0,\dots, \ell-1$ and define $\mathsf{g}(\ell n+i(k+1)+j)=j$ for $1\le j\le k+1$ and $i\ge 0$. It can be verified that this is a homomorphism from $H$ to $G$.

\section{Acknowledgment}
The authors want to thank Dr.\ Omid Etesami for constructive discussions.

%The authors want to thank Dr. Omid Etesami for constructive discussions.

%\bibliographystyle{amsalpha}

\bibliographystyle{ieeetr.bst}
\bibliography{mybibl}

\appendix

\section*{Appendix A}

We compute the value of $\alpha^*(C_n|C_m)$. We would like to prove that when $n,m>1$, then 
\[\alpha^*(C_n|C_m)=\begin{cases}
\frac{n}{m-1}& \text{if $n$ is even and $m$ is odd,} \\
\frac{n}{m}, &\text{if $m$ is even, }  \\
\frac{n}{m}, & \text{if $n$ is odd and $m$ is odd and $n\leq m$,}
\\
\frac{n}{m-1}, & \text{if $n$ is odd and $m$ is odd and $n> m$.}

\end{cases}
\]

By Lemma 2, $$\alpha^*(C_n|C_m)=\frac{n}{\alpha(C_n^c\boxtimes C_m)}\ge\frac{\alpha^*(C_n)}{\alpha^*(C_m)}= \frac{n}{m}$$
so $\alpha(C_n^c\boxtimes C_m)\le m$. 
Let $\{u_1,\dots, u_n\}$ and $\{v_1,\dots, v_m\}$ be the vertex sets of $C_n$ and $C_m$ respectively such that cyclically $u_i$ and $u_{i+1}$ and similarly $v_j$ and $v_{j+1}$ are connected. If $m\ge n$ and both are odd then we have a homomorphism $C_m\rightarrow C_n$ and hence by Lemma 5, $\alpha(C_n^c\boxtimes C_m)\ge m$ and so $\alpha^*(C_n|C_m)=\frac{n}{m}$. If $m$ is even then since $C_m$ is a perfect graph
$$\alpha^*(C_n|C_m)=\frac{\alpha^*(C_n)}{\alpha(C_m)}=\frac{n}{m}$$
Similarly if $n$ is even and $m$ is odd then
$$\alpha^*(C_n|C_m)=\frac{\alpha(C_n)}{\alpha(C_m)}=\frac{n}{m-1}$$
The only case remains when $n>m$, and they are both odd. In this case $(u_i,v_i)$ for $i=1,\dots, m-1$ is an independent set. So, we only need to show that we can not have $m$ independent vertices. Let $\mathcal S$ be a set of $m$ independent vertices of $C^c_n\boxtimes C_m$. There are two possibilities. Either, for each vertex $v_k$ of $C_m$, there is a vertex $(u_{i_k},v_k)$ in $\mathcal S$, or there is a vertex, say without loss of generality, $v_{m-1}$ in $C_m$ that there is more than one vertex in $\mathcal S$ with second component equal to $v_{m-1}$. The first case can not happen since, in that case, $u_{i_k}$ is cyclically connected to $u_{i_{k+1}}$ but $n>m$, so it can not happen. In the second case, for any two connected vertices $\{a,b\}$ in $C_m$, there are at most $2$ vertices in $\mathcal S$ with the second component equal to $a$ or $b$. Since for any three vertices in $C_n$, at least two are not connected. Therefore, if there are three vertices in $\mathcal S$ with second component $a$ or $b$, one can find two vertices among them that are connected, a contradiction. Hence, if there is more than one vertex in $\mathcal S$ with a second component $v_{m-1}$, there must be exactly two such vertices, and there are no vertex in $\mathcal S$ with a second component equal to $v_m$ or $v_{m-2}$. Now for each pair $\{v_{2i-1},v_{2i}\}$ for $i=1,\dots, \frac{m-3}{2}$ one has at most $2$ vertex in $S$ with second component equal to $v_{2i-1}$ or $v_{2i}$. So we have at most $2(\frac{m-3}{2})+2=m-1$ vertices in $\mathcal S$. This shows that $\alpha(C_n^c\boxtimes C_m)=m-1$ if $n>m$ and $n$ and $m$ are odd. Now, the calculation of $\alpha^*(C_n|C_m)$ is complete.

\section*{Appendix B}
In this appendix, we present an additional proof for Theorem \ref{thm1} by building upon the proof of Hales in \cite{hales1973numerical} for the fractional independence number and guided by the new form of the linear program. Consider the linear program in part (i) of Theorem \ref{thm1} where we maximize $\sum_{i=1}^k w_i$, subject to the constraints $\sum_{i=1}^kw_i |\mathcal T_i|\leq 1$ for any collection of sets $\mathcal T_1, \cdots, \mathcal T_k\in \mathcal I(H)$ such that $ \mathcal T_i$ and $\mathcal T_j$ are disconnected in $H$ if there is no edge between $v_i$ and $v_j$ in $G$. Let $\eta(G|H)$ be the maximum  of $\sum_{i=1}^k w_i$ in this linear program. We wish to show that
$\eta(G|H)=\alpha^*(G|H)$ where 
\[\alpha^*(G|H)=\sup_{W}\frac{\alpha(G\boxtimes W)}{\alpha(H\boxtimes W)}.\]
We first show that $\eta(G|H)\leq \alpha^*(G|H)$, and then show that $\eta(G|H)\geq \alpha^*(G|H)$.

\emph{Proof for $\eta(G|H)\geq \alpha^*(G|H)$:}
Take some arbitrary graph $W$. We need to show that
\begin{align}\eta(G|H)\geq \frac{\alpha(G\boxtimes W)}{\alpha(H\boxtimes W)}.\label{eqndk4n}\end{align}

    Let $\mathcal S$ be a maximum independent set for $G\boxtimes W$. For $v_i\in \mathcal V(G)$ define
    $$\mathcal S_i:=\{u\in \mathcal V(W): (v_i,u)\in \mathcal S\}.$$
    Note that if $v_i$ and $v_j$ are adjacent, then $\mathcal S_i$ and $\mathcal S_j$ are disconnected.  
    
    Consider the assignment $w_i=\frac{|\mathcal S_i|}{\alpha(H\boxtimes W)}$. Note that 
    \[\sum_i w_i=\frac{\alpha(G\boxtimes W)}{\alpha(H\boxtimes W)}.\]
    Thus, to show \eqref{eqndk4n}, it suffices to show that $w_i$'s satisfy the constraint of the linear program; that is for any subsets $\mathcal T_1,\dots, \mathcal T_n$ of independent sets of $H$ such that $\mathcal T_i$ and $\mathcal T_j$ are disconnected when there is no edge between $v_i$ and $v_j$ then 
    $$\sum_{i=1}^k w_i|\mathcal T_i|\le 1.$$
    Equivalently, we need to show that
    \begin{align}\sum_{i=1}^k |\mathcal S_i| |\mathcal T_i|\le \alpha(H\boxtimes W).\label{eqnkuh}\end{align}
    First, note that the subsets $\mathcal T_i\times \mathcal S_j$ of $\mathcal V(H\boxtimes W)$ are pairwise disjoint.   
    Indeed if $(x,y)$ belongs to $\mathcal T_{i}\times \mathcal S_{i}$ and $\mathcal T_{j}\times \mathcal S_{j}$ for $i\ne j$ then $\mathcal S_i$ and $\mathcal S_j$ are not disconnected and hence $v_i$ and $v_j$ are not adjacent and hence $\mathcal T_i$ and $\mathcal T_j$ are disconnected (giving a contradicition). 
    Next, observe that the union $\cup_{i=1}^n \mathcal T_i\times \mathcal S_i$ is an independent set of $H\boxtimes W$. To show this, take two vertices $(x_1,y_1)$ and $(x_2,y_2)$ in this set. If both of them belong to the same $\mathcal T_i\times \mathcal S_j$ then since $(v_i,y_1)$ and $(v_i,y_2)$ belong to the independent set $\mathcal S$ in $G\boxtimes W$ hence $y_1$ and $y_2$ are not adjacent in $W$. Therefore, the two vertices are not adjacent in $H\boxtimes W$. 
    
    Next, assume $(x_1,y_1)\in \mathcal T_{i}\times \mathcal S_{j}$ and $(x_2,y_2)\in \mathcal T_{j}\times \mathcal S_{j}$ and $i\ne j$. We are done if $y_1$ and $y_2$ are not adjacent. If $y_1=y_2$ then $v_i$ and $v_j$ are not adjacent and hence $\mathcal T_{i}$ and $\mathcal T_{j}$ are disconnected. Therefore, $x_1$ and $x_2$ are not adjacent. Finally, if $y_1$ and $y_2$ are adjacent, then $v_1$ and $v_2$ are not adjacent and again $\mathcal T_{i}$ and $\mathcal T_{j}$ are disconnected. So we constructed an independent set of $H\boxtimes W$ of size $\sum_{v\in \mathcal V(G)} |S_v| |T_v|$. So, \eqref{eqnkuh} is proved, and the proof for this part is complete.

\vspace{0.5cm}

\emph{Proof for $\eta(G|H)\leq \alpha^*(G|H)$:}

    Let $w_i$ be any rational numbers satisfying the constraints of the linear program. We will construct a finite graph $W$ such that 
    \begin{equation}
    \label{eqn1}
         \sum_{i=1}^k w_i\le \frac{\alpha(G\boxtimes W)}{\alpha(H\boxtimes W)}.   \end{equation}
   This would establish that  $\sum_{i=1}^k w_i\le \alpha^*(G|H)$. Since the maximum of $\sum_{i=1}^k w_i$ over all the rational solutions is the same as this maximum over all real numbers, we will be done with this part. 
   
   It remains to give the construction of $W$. Let $N$ be the common denominator of $w_1,\dots, w_k$. Let $W$ be a graph with $Nw_1+\dots+Nw_k$ vertices as follows: for each $v_i\in \mathcal V(G)$ take $\mathcal D_i$ be a set of vertices of $W$ consisting of $Nw_i$ vertices. There are no edges between the vertices in each $\mathcal D_i$. If $v_i$ and $v_j$ are not adjacent in $G$, then each vertex of $\mathcal D_i$ are connected by an edge to each vertex of $\mathcal D_j$; if $v_i$ and $v_j$ are adjacent in $G$ then there is no edge between any vertex of $\mathcal D_i$ and any vertex of $\mathcal D_j$. Thus, $W$ is essentially $G^c$ with the vertex $v_i$ being ``copied" $Nw_i$ times. 
   
   Now, it is clear that the set
    $$\{(v_i,y) : v_i\in \mathcal V(G), y\in \mathcal D_i\}$$
    is an independent set in $G\boxtimes W$ of size $Nw_1+\dots+Nw_k$ and hence
    \begin{equation}
    \label{eqn2}
        Nw_1+\dots+Nw_k\le \alpha(G\boxtimes W).    
        \end{equation}
    
    Take a maximal independent set $\mathcal T$ for $H\boxtimes W$ and for $y\in\mathcal V(W)$ let
    $$\mathcal T_y=\{u\in \mathcal V(H): (u,y)\in \mathcal T\}.$$
Note that $\mathcal T_y$ is an independent set in $H$.
    Some observations are in order: 
    Note that if $y_1$ and $y_2$ belong to $\mathcal D_i$ then $\mathcal T_{y_1}=\mathcal T_{y_2}$. Otherwise we can make $\mathcal T$ to a bigger independent set by adding $(u,y_1)$ and $(u,y_2)$ with $u\in \mathcal T_{y_1}\cup \mathcal T_{y_2}$ to $\mathcal{T}$. Since for $y_1$ and $y_2$ belonging to $\mathcal D_i$, we have $\mathcal T_{y_1}=\mathcal T_{y_2}$, we can use the notation $\mathcal T_i$ for $\mathcal T_y$ for any $y\in \mathcal D_i$. Therefore, assuming $\mathcal{V}(G)=\{1,2,\cdots, k\}$, we get the sets $\mathcal T_1, \mathcal T_2, \cdots, \mathcal T_k$. 

If $v_i$ and $v_j$ are not connected in $G$, then vertices in the sets $\mathcal{D}_i$ and $\mathcal{D}_j$ are connected in $W$ and hence $\mathcal T_i$ and $\mathcal T_j$ must be disconnected in $H$ (otherwise $\mathcal{T}$ will not be an independent set for $H\boxtimes W$). Therefore, since the numbers $w_1,\dots, w_k$ satisfy the constraints of our linear program, we get that
$\sum_{i=1}^k w_i|\mathcal T_i|\le 1$.
Now since $|\mathcal{D}_i|=Nw_i$ we have
    \begin{equation}
    \label{eqn3}
\alpha(H\boxtimes W)=|\mathcal T|=\sum_{y\in\mathcal V(W)}|\mathcal T_y|=\sum_{i=1}^k Nw_i|\mathcal T_i|\le N.
    \end{equation}
      
    From inequalities \eqref{eqn2} and \eqref{eqn3}, the inequality \eqref{eqn1} follows.

\begin{remark}
    Note that the construction of the graph $W$ depends on graph $H$ only through the weights $w_i$ of the linear program. 
\end{remark}

\end{document}